\documentclass[letter]{amsart}
\usepackage{amssymb,latexsym,amsfonts,amsmath,amsthm}
\usepackage{graphics,graphicx,epic,shadow}
\usepackage{dsfont}
\usepackage{enumerate}
\newtheorem{theorem}{Theorem}[section]

\newtheorem{definition}[theorem]{Definition}

\numberwithin{equation}{section}

\newcommand{\Z}{\mathbb{Z}}
\newcommand{\R}{\mathbb{R}}
\newcommand{\N}{\mathbb{N}}
\newcommand{\C}{\mathbb{C}}

\newcommand{\supp}{\operatorname{supp}}

\renewcommand{\S}{\mathcal{S}}

\newcommand{\ident}{w}

\begin{document}
\allowdisplaybreaks
\title{ 
Local sampling and approximation of operators with bandlimited Kohn-Nirenberg symbols
}
\author{Felix Krahmer  and G\"{o}tz E. Pfander}{\let\thefootnote\relax\footnotetext{Felix Krahmer is with the University of G\"ottingen, Institute for Numerical and Applied Mathematics, Lotze\-stra\ss e 16-18, 37083 G\"ottingen, Germany, Tel.:\,+49 551 39 10584, Fax:\,+49 551 39 3944, f.krahmer@math.uni-goettingen.de.

G\"otz E. Pfander is with Jacobs University Bremen, School of Engineering and Science, Campus Ring 12, 28759 Bremen, Germany.
Tel.: +49 421 200 3211, Fax:\,+49 421 200 49 3211, g.pfander@jacobs-university.de.
} }
\date{\today}
\maketitle

\begin{abstract} 
Recent sampling theorems allow for the recovery of operators with bandlimited Kohn-Nirenberg symbols from their response to a single discretely supported identifier signal. The available results are inherently non-local.  For example, we show that in order to recover a bandlimited operator precisely, the identifier cannot decay in time nor in frequency. Moreover, a concept of local and discrete representation is missing from the theory.
In this paper, we develop tools that address these shortcomings.

We show that to obtain a local approximation of an operator, it is sufficient to test the operator on a truncated and mollified delta train, that is, on a compactly supported Schwarz class function. 
To compute the operator numerically, discrete measurements 
can be obtained from the response function which are localized in the sense that a local selection of the values  yields a local approximation of the operator. 

Central to our analysis is to  conceptualize the meaning of localization for operators with bandlimited Kohn-Nirenberg symbol.\\[.2cm]
\noindent
{\it Keywords.}\ \  Operator identification, pseudodifferential operators, Kohn-Nirenberg symbol, time frequency localization, local approximation, tight Gabor frames.\\[.1cm]
{\it 2010 Mathematics Subject Classification.} \ \ Primary 41A35, 94A20; Secondary 42B35, 47B35, 47G30, 94A20.
\end{abstract}


\section{Introduction}

In {\em communications engineering}, the effect of a slowly time-varying communication channel is commonly modeled as superposition of  translations (time shifts due to multipath propagation) and modulations (frequency shifts caused by Doppler effects). In order to recover transmitted signals from their channel outputs, precise knowledge of the nature of the channel is required. A common procedure for channel identification in this sense is to periodically send  short duration test signals. The resulting outputs are then used to estimate channel parameters which allow for an inversion of the operator \cite{ka62,Bel69,Pfander06-3,Pfander06-2,BGE11,HB12}.

Kailath \cite{ka62} and Bello \cite{Bel69} analyzed the identifiability of such channels. In mathematical terms, the channels considered are characterized by bandlimited Kohn-Nirenberg symbols  
and the channel identification problem becomes an operator identification problem: can an operator with bandlimited Kohn-Nirenberg symbol be identified from the output corresponding to a given test input signal?

Kozek and Pfander \cite{Pfander06-3}, and Pfander and Walnut \cite{Pfander06-2} gave mathematical proof of the assertions by Kailath and Bello that there exists a suitable test signal as long as the band support of the symbol of the operator has outer Jordan content less than one. The suggested test signals  are periodically weighted regularly spaced Dirac-delta distributions as introduced in  \cite{Pfander06-2}. In \cite{Pfa10}, Pfander coined the term  {\em operator sampling} as the resulting theory has many direct parallels to the sampling theory for bandlimited functions. For example, an operator sampling reconstruction formula was established which generalizes  the reconstruction formula in the classical sampling theorem for bandlimited functions (see \cite{Pfa10} and Theorem~\ref{thm:main-simple} below).  

As the test signals which appear in the results of \cite{Pfander06-3,Pfander06-2,Pfa10,PW12} decay neither in time nor in frequency, they cannot be realized in practice. In this paper, we show that indeed,  for stable identification of operator classes defined by a bandlimitation of the Kohn-Nirenberg symbol, test signals that lack decay in time and frequency are necessary.   When seeking to recover only the operator's action on a time-frequency localized subspace, however, this ideal but impractical signal can be replaced with a mollified and truncated copy; the test signal can thereby be chosen to be a compactly supported Schwartz function as shown below.  
 
Furthermore, an important difference to the sampling theory for bandlimited functions is that the response to a test signal in operator sampling is a square-integrable function rather than a discrete set of sample values. While there are many ways to discretely represent the response function, the question remains which of the multitude of commonly considered representations allow to recover the operator most efficiently. In the case of a bandlimited function, one  feature that distinguishes the representation by samples is locality: a sample is the function value at a given location;  due to the smoothness of bandlimited functions it represents the function in the neighborhood of the sampling point. A key consequence of this feature is that it allows to approximate the function in a given region using only samples taken in  a fixed-size neighborhood of it. 

In this paper we develop discrete representations of operators with bandlimited Kohn-Nirenberg symbols that, on the one hand, can be computed in a direct and simple way from the output corresponding to a test signal and, on the other hand, have  locality properties analogous to those we appreciate in the classical sampling theory. We work with the same concept of locality as in the localized sampling results mentioned above, namely, locality will be defined through the action of the operator on time-frequency localized functions. Combining the two parts, we obtain that time-frequency measurements of the output corresponding to a truncated and mollified weighted sum of Dirac delta distributions yield a local discrete representation of a bandlimited operator. 

The paper is organized as follows. In Section~\ref{sec:background} we recall operator sampling terminology in some detail and discuss previous results. We then summarize our main results in Section~\ref{sec:results}.  Section~\ref{sec:local} provides results on local approximations of operators; in Section~\ref{section:identifier} we discuss identification using smooth and finite duration test signals, and Section~\ref{sec:Op_id} uses Gabor frames to derive our novel discretization scheme for operators with bandlimited Kohn-Nirenberg symbols.

\section{Background}\label{sec:background}
\subsection{Symbolic calculus} 
The Schwartz kernel theorem states that every continuous linear operator \hbox{$H:\mathcal{S}(\R)\to\mathcal{S}'(\R)$} is of the form $$Hf(x) = \int  \kappa (x,t)f(t)dt$$ for a unique {\it kernel} $\kappa\in \mathcal{S}'(\R^{2})$, where $\mathcal{S}(\R^d)$ is the Schwartz space, and $\mathcal{S}'(\R^d)$ is its dual, the space of tempered distributions \cite{Folland89}. This integral representation is understood in the weak sense, that is, $$ \langle Hf,g\rangle = \langle \kappa,\overline f{\otimes} g\rangle $$ for all $f,g\in \mathcal{S}(\R)$, where $f{\otimes} g (x,y) = f(x)g(y)$ and $\langle\, \cdot\,,\,\cdot\,\rangle$ is the sesquilinear pairing between $\mathcal{S}(\R)$ and $\mathcal{S}'(\R)$, and as $\mathcal{S}(\R)$ continuously embeds into $L^2(\R)$, Schwartz kernel representations exist in particular for bounded operators on $L^2(\R^d)$.

Each such operator has consequently a {\it spreading function} representation 
\begin{equation}\label{eqn:spreadingrepresentation}
 Hf(x) = \iint \eta(t,\gamma) e^{2\pi i \gamma x} \, f(x-t)\ \,dt\, d\gamma,
\end{equation}
a {\em time-varying impulse response} representation
\begin{equation*}
Hf(x)=\int    h(x,t) f(x-t) dt,
\end{equation*}
and a {\it Kohn-Nirenberg symbol} representation 
\begin{equation}\label{eqn:knrepresentation}
Hf(x) = \int  \sigma(x,\xi)\widehat{f}(\xi)\,e^{2\pi ix\xi}\ d\xi,
\end{equation}
where the Fourier transform  $\widehat f$ is  normalized by  $$\mathcal F f(\xi)=\widehat f(\xi) = \int  f(t)\, e^{-2\pi i t\xi}\, dt$$ for integrable $f$. 

We write $H_{\sigma}$ and $\sigma_H$, $\eta_{H}$, $\kappa_{H}$ when it is necessary to emphasize the correspondence between $H$ and $\sigma$, $\eta$, $\kappa$. The symbols $\sigma$ and $\eta$ are related via the {\it symplectic Fourier transform } 
 $\mathcal F_s$ which  is defined densely by
\begin{equation*}
\mathcal F_s \sigma(t,\gamma)=\iint\sigma(x,\xi) e^{-2\pi i(x\gamma -t\xi)}\,dx\, d\xi\,,
\end{equation*}
that is, $\sigma=\mathcal F_s \eta$.

For convenience, we use the notation $\boldsymbol{\eta}_H(t,\gamma) = e^{2\pi i\gamma t}\eta_H(t,\gamma)$ and $\boldsymbol{\sigma}_H$ for its symplectic Fourier transform. A straightforward computation shows $\sigma_{H^\ast} =  \overline{\boldsymbol{\sigma}_H},
$
where $H^{\ast}$ denotes the adjoint of $H$. In our proofs, we shall frequently transition from  $\sigma$ to $\boldsymbol \sigma$. This does not cause a problem in our analysis since inequality \eqref{eqn:equivalentnorms} below combined with $\|H\|_{{\mathcal L}({L^2(\R)})}=\|H^\ast \|_{{\mathcal L}({L^2(\R)})}$ shows that for $M\subseteq \R^2$ compact there exist $A,B>0$ with 
\begin{equation}\label{eqn:iso}
A \|\sigma\|_{L^\infty(\R^2)}\leq \|\boldsymbol \sigma\|_{L^\infty(\R^2)}\leq B\|\sigma\|_{L^\infty(\R^2)}
\end{equation}
for all $H_\sigma \in OPW(M)$.

\subsection{Sampling in operator Paley Wiener spaces}
\quad
Following \cite{ka62, Bel69, Pfander06-3,PW12, Pfa10}, the operators considered in this paper are assumed to have strictly bandlimited Kohn-Nirenberg symbols, that is, they have compactly supported spreading function. Slowly time-varying mobile communications channels may violate this assumption \cite{Strohmer06}; a more refined model is that the spreading function exhibits rapid decay.  Still this suffices to guarantee that truncating the spreading function introduces a global error that can be controlled. For example, applying Theorem~\ref{thm:main-simple} to an operator $H$ whose spreading function has $L^2$ distance  at most $\epsilon$ to a function supported on a rectangle of area one
results in an operator which differs from $H$  differs by at most $2\epsilon$ in operator norm (see \cite{KraPfa13b} for further details). This justifies to restrict to the simpler model of strictly bandlimited symbols.

The space of bounded operators  whose Kohn-Nirenberg symbols are bandlimited to a given set $M$ --- we will also use the shorthand terminology {\em bandlimited operators} ---  is called {\em operator Paley-Wiener space}\footnote{In general terms, operator Paley-Wiener spaces are defined by requiring its members to have bandlimited Kohn-Nirenberg symbols which are in a  prescribed weighted and mixed $L^p$ space \cite{Pfa10}. For example, to restrict the attention to bandlimited Hilbert-Schmidt operators, we would consider only operators with square integrable symbols. These form a subset of the operators considered in this paper.}; it is denoted by
$$ OPW(M) = \{H\in \mathcal L (L^2(\R)):\  \ \supp \mathcal F_s \sigma_H\, \subseteq M\}. $$ 

The Kohn-Nirenberg symbol of an $L^2$-bounded operator with   $\supp \mathcal F_s \sigma_H$ compact is bounded. In fact,  for some $A,B>0$ we have,
\begin{align}
 \label{eqn:equivalentnorms}
 A\|\sigma_H \|_{L^\infty(\R)}\leq \|H\|_{\mathcal L(L^2(\R))}\leq B \|\sigma_H \|_{L^\infty(\R)}\,, \quad H\in OPW(M),
\end{align}
where $\|H \|_{\mathcal L(L^2(\R))}$ is the operator norm of $H$  (Proposition~\ref {thm:normequiv} below). 

\vspace{.2cm}
Certainly, if we have direct access to $\sigma_H$, then some of our approximation theoretic goals can be accomplished using classical two-dimensional sampling results applied to $\sigma_H$.  In the model considered here, however, we do not have access to any of the values of the symbol $\sigma_H$ of the operator $H$ directly, but we must rely on the operator output $H\ident $ which results from applying $H$ to a single test input $\ident$. Due to stability consideration, we say that the linear space $OPW(M)$ is identifiable by $\ident$ if for $A,B>0$ we have
\begin{align}
 \label{eqn:identifable}
 A\|H\|_{\mathcal L(L^2(\R))}\leq \|H\ident\|_{L^2(\R)} \leq B\|H\|_{\mathcal L(L^2(\R))}\,, 
\end{align}
for all $  H\in OPW(M)$ \cite{Pfander06-3}.
``Sampling'' the operator means that the identifier $\ident$ in \eqref{eqn:identifable}  is a weighted sequence of Dirac delta distributions, that is,
\begin{equation*}
\ident = \sum_{k\in\Z}c_k\delta_{kT},
\end{equation*}
where  $c_k$ is an appropriately chosen periodic sequence \cite{lapfwa05,Pfander06-2,Pfa10}.

A guiding paradigm in the sampling theory of operators is the direct analogy to  sampling of bandlimited functions. To illustrate this analogy, we compare the classical sampling theorem (often credited to Cauchy, Kotelnikov, Shannon, and Whittaker, among others), Theorem~\ref{thm:classical}, with the corresponding result for operators, Theorem~\ref{thm:main-simple} \cite{Pfa10}. Note that Theorem~\ref{thm:classical} formally follows from Theorem~\ref{thm:main-simple} by choosing the operator $H$ in Theorem~\ref{thm:main-simple} to be the pointwise multiplication operator $f\mapsto \sigma \cdot f$ \cite{Pfa10}.

The engineering intuition underlying sampling theorems is that reducing a function  to periodic samples at a rate of $1/T$ samples per unit interval corresponds to a periodization with shift $1/T$ in frequency space \cite{OppSchBu99}. Thus, as long as $ T\Omega \leq 1$, a function bandlimited to $[-\frac \Omega 2, \frac \Omega 2]$ can be recovered via a convolution with a low-pass kernel, that is, a function $\phi$ that satisfies
\begin{equation}\label{eq:lowpass}
 \widehat \phi (\xi) = \begin{cases} 1/\Omega,  \quad &\text{if } |\xi|\leq \frac{\Omega}{2}\,,\\ 
                         0, \quad &\text{if } |\xi|\geq \frac{1}{2T}\,.
                       \end{cases}
\end{equation}
If  $T\Omega=1$, the only such function is the sinc kernel
$\phi(t)=    \operatorname{sinc}(\pi t/T)=   \frac{\sin (\pi  t/T)}{ \pi  t/T}$. For $T\Omega<1$, there are many such functions; in particular $\phi$ in the Schwartz class is possible. With this notion, the classical sampling theorem reads as follows.

\begin{theorem}\label{thm:classical}
  For $g\in  L^2(\R)$ with $\supp \mathcal F g\subseteq [-\frac \Omega 2, \frac \Omega 2]$ and  $T\Omega\leq 1$, we have
  \begin{align}
    g(x)=\sum_{n\in\Z} g(nT)\, \phi (x-nT)\label{eqn:functionreconstruction-simple}
  \end{align}
  with uniform convergence and convergence in $L^2(\R)$. 
Here, $\phi$ is any low-pass kernel satisfying \eqref{eq:lowpass}.
\end{theorem}

Recall that every operator $H$ on $L^2(\R)$ is in one-to-one correspondence with its kernel $\kappa_H$, that is, for a unique tempered distribution $\kappa_{H}$, we have $Hf(x)=\int \kappa_H(x,y)\,f(y)\,dy$ weakly.  In the following, $\chi_A$ denotes the characteristic function of a set $A$.

\begin{theorem}\label{thm:main-simple} \hspace{-.1cm}\cite{Pfa10}\hspace{.1cm}
  For $H: L^2(\R)\longrightarrow L^2(\R)$ with $\sigma_{H}\in L^2(\R^2)$, $\supp \mathcal F_s \sigma_H \subseteq [0, T] {\times}[-\frac \Omega 2, \frac \Omega 2]$, and  $T\Omega\leq 1$, we have
  \begin{align}
    \kappa_H(x+t,x)= \chi_{[0,T]}(t)\,\sum_{k\in\Z} \big(H\sum_{n\in\Z}\delta_{nT}\big)(t+kT)\, {\phi(x-kT)} \label{eqn:operatorreconstruction-simple},
  \end{align}
  with convergence in $L^2(\R^2)$ and uniform convergence in $x$ for each $t$. Again,  $\phi$ is any low-pass kernel satisfying \eqref{eq:lowpass}.
\end{theorem}

We point to an important difference between the applicability of Theorems~\ref{thm:classical}  and ~\ref{thm:main-simple}: in Theorem~\ref{thm:classical}, a bandlimitation   to a large set $[-\frac \Omega 2, \frac \Omega 2]$ can be resolved by choosing  a small $T$; on the other side, Theorem~\ref{thm:main-simple} is not applicable if the bandlimiting set $[0, T] {\times}[-\frac \Omega 2, \frac \Omega 2]$ has area greater than one. Indeed, the  the following is known.

\begin{theorem}  \hspace{-.1cm} \cite{Pfander06-2,Pfa08}  \hspace{.1cm} $OPW(M)$ is identifiable in the sense of \eqref{eqn:identifable} with appropriate $w=\sum_{n\in\Z}c_n \delta_{nT}$ if $M$ is compact with measure less than 1. If $M$ is open and has area greater than 1, then there exists no tempered distribution $w$ identifying $OPW(M)$.
\end{theorem}

Hence, it is necessary to  restrict ourselves to operator Paley-Wiener spaces defined by compact sets $M$ with Lebesgue measure one. For such spaces, one can extend Theorem~\ref{thm:main-simple} to  the following.

\begin{theorem}\label{thm:reconstruction} \hspace{-.1cm}\cite{PW12} \hspace{.1cm}
Let $M$ be compact with Lebesgue measure less than one. Then there exists $T,\Omega>0$ with $ {T\Omega}=\frac 1 L$, $L$ prime, $\delta>0$, and  $L$-periodic sequences $\{c_n\}_n$, $\{b_{jq}\}_q$, $j=0,1,\ldots,n-1$, so that 
\begin{align} \label{eqn:reconstructionformula}
\kappa_H(x+t,x) &= LT  \sum_{j=0}^{L-1} r(t-k_jT) e^{2\pi i n_j \Omega x} \sum_{q\in\Z} b_{jq} \\&\notag \qquad \qquad \big( H \sum_n c_n\delta_{nT}\big) (t- (k_j-q)T) \, \phi(x+(k_j-q)T) \,, \quad  H\in OPW(M),
\end{align}
where  $r,\phi $ are Schwartz class functions that satisfy
\begin{align*}
r(t)\widehat{\phi}(\gamma)&= 0 \textnormal{  if  } \ (t,\gamma) \notin (-\delta, T + \delta)\times (-\delta,\Omega+\delta),
\end{align*}
and
\begin{equation}\label{eq:r_phi_2}
\sum_{k\in\Z} r(t - kT) \equiv 1 \equiv \sum_{n\in\Z} \widehat{\phi}(\gamma - n\Omega)\,.
\end{equation}
Moreover, \eqref{eqn:reconstructionformula} converges in $L^2(\R)$ and uniformly in $x$ for each $t$.
\end{theorem}

 \section{Main Results}\label{sec:results}

\subsection{Local representations of operators} In classical as well as in operator sampling, working with Schwartz class kernels $r$,$\phi$ 
is of advantage. Indeed, in the classical sampling theorem, the slow decay of the sinc kernel in \eqref{eqn:functionreconstruction-simple} implies that a small perturbation of just a few coefficients $g(nT)$ can lead to significant deviations of all values $g(t)$ outside of the sampling grid $T\Z$; this includes values achieved at locations far from the sampling points $nT$.  Hence to approximately recover the function values locally, that is, on an compact interval, it does not suffice to know the function samples in a constant size neighborhood of that interval. 
When working with Schwartz class kernels, in contrast, such a local approximate reconstruction is possible; one can achieve
\begin{align}
    \big| g(x) -   \sum_{nT\in  [a,b]} g(n T)\, \phi( x-nT) \big|<\epsilon \label{eqn:approximatefunctionreconstruction},
  \end{align}
for all $x\in [a+d(\epsilon), b-d(\epsilon)]$ where the neighborhood size $d(\epsilon)$ depends on the approximation level $\epsilon$ but  not  on the interval $[a,b]$.

A corresponding possibility of using local information for local reconstruction is not given  in Theorem~\ref{thm:main-simple}.  Moreover, the identifier $w=  \sum_{n\in\Z}\delta_{nT}$ neither decays in time or in frequency, clearly showing that in practice this input signal is not usable. However, in the framework of Theorem~\ref{thm:main-simple}, this is unavoidable, as we show in the following theorem.

\begin{theorem}\label{thm:nodecay}
 If the tempered distribution $w$ identifies $OPW([0,T]{\times} [-\Omega/2,\Omega/2])$, $T\Omega>0$, then $w$ decays weakly neither  in time nor in frequency, that is, we have neither
$$
 \langle w , \varphi(\cdot-x)\rangle \stackrel{x\to\pm\infty}  \longrightarrow 0 \quad \text{ nor }  \quad \langle \widehat w , \varphi(\cdot-\xi)\rangle \stackrel{\xi\to\pm\infty}  \longrightarrow  0$$ for all Schwartz class functions $\varphi$.
\end{theorem}

We address this problem  by developing a concept of ``local recovery'' of an operator, in analogy to the local recovery of a function in  \eqref{eqn:approximatefunctionreconstruction}.
Indeed, the key to most results presented in this paper is to aim only for the recovery of the operator restricted to a set of functions ``localized'' on a prescribed set $S$ in the time-frequency plane. This  is indeed reasonable in communications where  band and time constraints on transmitted signals are frequently present. In \cite{HB12}, for example, operators  that map  bandlimited input signals to finite duration output signals are considered. Bivariate Fourier series expansions of such an operator's compactly supported Kohn-Nirenberg symbol allow  the authors to discretize the a-priori continuous input-output relations \eqref{eqn:knrepresentation} and \eqref{eqn:spreadingrepresentation}.

Our definition of function localization in time and frequency is based on Gabor frames. It involves translation and modulation operators, $$\mathcal T_t f: f\mapsto f(\cdot -t)\text{ and } \mathcal M_\nu :f\mapsto e^{2\pi i \nu(\cdot)}f.$$ These operators are unitary on $L^2(\R)$ and isomorphisms on all function and distribution spaces considered in this paper.

For any $g\in L^2(\R)$ and $a,b>0$,  we say that the Gabor system
$$(g,a\Z \times b\Z)=\{\mathcal T_{ka} \mathcal M_{\ell b} g \}_{k,\ell\in\Z}$$ is a tight frame for $L^2(\R^d)$ if for some $A>0$, the so-called frame bound, we have
$$
    f=A\,\sum_{k,\ell\in\Z} \langle f, \mathcal T_{ka}\,\mathcal M_{\ell b} g\rangle\ \mathcal T_{ka}\,\mathcal M_{\ell b} g$$
    for all $f\in L^2(\R^d)$.
Each coefficient in this expansion can be interpreted to reflect the local behavior of the function near the indexing point in time-frequency space. Hence, a natural way to define time-frequency localized functions is that all but certain expansion coefficients are small. 

\begin{definition}\label{def:localization}
Let $(g,a\Z\times b\Z)$, $g\in \mathcal S(\R)$, be a tight frame for $L^2(\R)$ with frame bound 1. We say that $f\in L^2(\R)$  is $\epsilon$--time-frequency localized on the set $S$ if
  \begin{align*}
    \sum_{ (ka,\ell b) \in  \normalsize S } |\langle f,\mathcal M_{\ell b}\mathcal T_{ka} \,g\rangle|^2\geq (1-\epsilon^2)\,\sum_{(ka,\ell b) \in  \normalsize  \R^2} |\langle f, \mathcal M_{\ell b}\mathcal T_{ka}\,g\rangle|^2 \,.
    \end{align*}
\end{definition}

Our next result states that a sufficient condition for two operators to approximately agree on  functions $\epsilon$--time-frequency localized on a set $S$ is that their Kohn-Nirenberg symbols almost agree on a neighborhood of $S$.
Below, $B(r)$ denotes the Euclidean unit ball with radius $r$ and center $0$; the dimension will always be clear from the context. For brevity of notation, we set  $S-B(r)=\big(S^c+B(r)\big)^c$ for $S\subseteq \R^2$.

\begin{theorem}\label{thm:localization1}
Fix $M$ compact  and let $(g,a\Z\times b\Z)$, $g\in \mathcal S(\R)$, be a tight frame for $L^2(\R)$ with frame bound 1.  Then any pair of operators $H,\widetilde H \in OPW(M)$ for which  one has
  $$
   \|\sigma_H\|_{L^\infty(\R^2)},\  \|\sigma_{\widetilde H}\|_{L^\infty(\R^2)}\leq \mu\quad\text{and}\quad \|\sigma_H-\sigma_{\widetilde H}\|_{L^\infty(S)}\leq \epsilon\,\mu
  $$
  on a set $S\subseteq \R^2$ satisfy
  \begin{align*}
    \|Hf-\widetilde H f\|_{L^2(\R)}\leq C\,\epsilon\,\mu\,  \|f\|_{L^2(\R)}
  \end{align*}
  for all $f\in L^2(\R)$  that are $\epsilon$--time-frequency localized on  $S-B\big(d(\epsilon)\big)
  $ in the sense of Definition~\ref{def:localization}. Here $C>0$ is an absolute constant and $d:(0,1)\longrightarrow \R^+$ satisfies $d(\epsilon)=o(\sqrt[k]{1/\epsilon})$ for all $k\in\N$ as $\epsilon\to 0$.
\end{theorem}

A generalization of Theorem~\ref{thm:localization1} -- labeled Theorem~\ref{prop:exp-general} --  is proven in Section~\ref{sec:local}. 

Our next main result concerns truncated and mollified versions of the identifier $\sum_n c_n \delta_{nT}$ and provides localized versions of Theorems~\ref{thm:main-simple} and \ref{thm:reconstruction}.  For $S=\R^2$, it reduces to Theorems~\ref{thm:main-simple} and  \ref{thm:reconstruction}.

\begin{theorem}\label{thm:trunc}
 Fix $M$ compact with Lebesgue measure $\mu(M)<1$ and let $(g,a\Z\times b\Z)$, $g\in \mathcal S(\R)$, be a tight frame for $L^2(\R)$ with frame bound 1.
 Let  $S\subseteq I_1\times I_2 \subseteq \R^2$, where  $I_1$ and $I_2$ may coincide with $\R$. Furthermore, choose the tempered distribution $\varphi$ such that $\varphi\geq 0$ and $\widehat\varphi\equiv 1$ on $I_2$ and let 
\begin{equation*}
\widetilde\ident=\sum_{nT\in I_1} c_n \varphi(\cdot-nT).
\end{equation*}
 Then for any $H\in OPW(M)$ with
  $ \|\sigma_H\|_{L^\infty(\R^2)}\leq \mu$ and  
 $\widetilde H\in OPW(M)$ defined via
\begin{align}
\kappa_{\widetilde H}(x+t,x) &= LT  \sum_{j=0}^{L-1} r(t-k_jT)\Big( \sum_{q\in\Z} b_{jq} H\widetilde\ident(t- (k_j-q)T) \phi(x+(k_j-q)T)\Big)\, e^{2\pi i n_j \Omega x}, 
\label{eqn:reconstructionformulaTilde}
\end{align}
one has
  \begin{align*}%
    \label{thm:localization}
    \|Hf- \widetilde H f\|_{L^2(\R)}\leq C\,\epsilon\,\mu\,  \|f\|_{L^2(\R)}
  \end{align*}
  for all $f\in L^2(\R)$  that are $\epsilon$--time-frequency localized on  $S-B\big(d(\epsilon)\big)
$ in the sense of Definition~\ref{def:localization}.
Here $C>0$ is an absolute constant, $r$ and $\phi $ are Schwartz class functions defined as in Theorem~\ref{thm:reconstruction}, but for $\delta>0$ such that $\mu(M+[-3\delta,3\delta]^2)<1$, and $d:(0,1)\longrightarrow \R^+$ is a function independent of $S$  which satisfies $d(\epsilon)=o(\sqrt[k]{1/\epsilon})$ for all $k\in\N$ as $\epsilon\to 0$.
\end{theorem}

For rectangular bandlimitation domains $M= [0, T] {\times}[-\frac \Omega 2, \frac \Omega 2]$ one can choose the   identifier $\sum\limits_{nT\in I_1} \varphi(\cdot-nT)$ and define $\widetilde H$ via the   formula
  \begin{align*}
    \kappa_{\widetilde H}(x+t,x)=T\sum_{n\in\Z} \big(H \sum_{nT\in I_1} \varphi(\cdot-nT)\big)
(t+nT)\, {\phi(x-nT)}\,.
  \end{align*}

Note that this theorem is completely analogous to the condition~\eqref{eqn:approximatefunctionreconstruction} for localized function sampling. Due to the  two-dimensional nature of the operator, however, localization is an issue in both time (restricting to a finite number of deltas) and frequency (replacing the deltas by approximate identities). If one is interested in localization only in time or only in frequency, one can choose one of the $I_i$ to be $\R$ and thus consider
\begin{equation*}
  \widetilde\ident=\sum_{nT\in I_1} c_n\delta_{nT} \quad \text{or} \quad \widetilde\ident=\sum_n c_n \varphi(\cdot-nT),  
\end{equation*}
again with $(c_n)\equiv 1$ in case of rectangular domains $M$.

\subsection{Local sampling of operators}

An additional important structural difference between classical sampling and operator sampling remains: in Theorems~\ref{thm:main-simple} and \ref{thm:trunc}, the
reconstruction formulas \eqref{eqn:operatorreconstruction-simple} and \eqref{eqn:reconstructionformula} involve as    ``coefficients''  functions, not scalars. Among the many possibilities to discretely represent the operator's response to the identifier $w$,
we consider Gabor representations of this sample function. A time-frequency localized subset of the coefficients will then yield a corresponding local approximation of the operator. Theorem~\ref{prop:exp} below establishes a reconstruction formula based on Gabor coefficients  that allows for the exact recovery of the operator; Theorem~\ref{thm:localsamp} shows that a local subset of the coefficients yields a local approximation of the operator. Again, one obtains considerably simpler formulas for rectangular domains, but for reasons of brevity, we focus on the comprehensive setup of arbitrary domains.

For a Schwartz class function $\phi$ and a tempered distribution $f$ on $\R$ we call
\begin{equation*}
 V_\phi f(x,\xi)=\langle f, \mathcal M_\xi  \mathcal T_x \phi\rangle, \quad x,\xi\in\R,
\end{equation*}
the {\em short-time Fourier transform} of $f$ with respect to the window function $\phi$. Throughout this paper, all pairings $\langle \cdot, \cdot \rangle$ are taken to be linear in the first component and antilinear in the second.

\begin{theorem}\label{prop:exp}
For $M$ compact with Lebesgue measure $\mu(M)<1$ there exists $L$ prime, $\delta>0$, $T,\Omega>0$ with $T\Omega=1/L$, and $L$-periodic sequences $\{c_n\}_n$, $\{ b_{jq}\}_q$, $j=0,\ldots,L-1$, so that for $H\in OPW(M)$, 
%
\begin{align}\label{eq:propexp}
 \sigma_H(x,\xi)
= \frac{LT}{\beta_1\beta_2} & \sum_{j=0}^{L-1} e^{-2\pi i(xn_j\Omega + \xi  k_j T )} e^{2\pi i n_j\Omega k_j T} \\ &\quad \notag \sum_{m,\ell\in\Z} \!\sigma^{(j)}_{m,\ell} \ V_\phi r\big(x \!-\! \big(\frac{m  L}{\beta_1}+k_j\big) T, \ \xi\!-\!\big(\frac{\ell L }{\beta_2}+n_j\big)\Omega\big),\end{align}
where
\begin{align*}
\sigma^{(j)}_{m,\ell} = \sum_{q\in\Z}\  b_{jq}\ \phi \big((-q -k_j-mL/\beta_1)T\big)\ \langle H\sum_n c_n\delta_{nT} ,\ \mathcal T_{qT}\mathcal M_{\ell \Omega L /\beta_2}\, r\rangle,
\end{align*}
and
$r,\phi $ are Schwartz class functions such that $r$ and $\widehat{\phi}$ are real valued and satisfy\footnote{For example, we can choose $r=\chi_{[0,T)}\ast \varphi_\delta$, where $\varphi_\delta$ is an approximate identity, that is, a non-negative function with $\varphi_\delta\in\mathcal S(\R)$, $\supp \varphi_\delta\subseteq [-\delta/2,\delta/2]$, and $\int\varphi_\delta=1$.}
\begin{align}\label{eq:r_phi_1a}
r(t) = 0 \textnormal{  if  } \ t \notin (-\delta, \delta + T), \quad
\widehat{\phi}(\gamma)=0 \textnormal{  if  } \ \gamma\notin (-\delta-\Omega/2,\delta+\Omega/2),
\end{align}
and
\begin{equation}\label{eq:r_phi_2a}
\sum_{k\in\Z} |r(t + kT)|^2 \equiv 1 \equiv\sum_{n\in\Z} |\widehat{\phi}(\gamma + n\Omega)|^2,
\end{equation}
with oversampling rates
$\beta_2\geq 1 +2\delta /T$ and $\beta_1\geq 1 +2\delta /\Omega$.\footnote{Then the Gabor systems $\{r_{k,l}=\mathcal T_{kT}\mathcal M_{\ell/\beta_2 T}\,r\}_{k,\ell\in\Z}$, $\{\mathcal T_{n\Omega}\mathcal M_{m/\beta_1 \Omega}\, \widehat \phi \}_{m,n\in\Z}$, and $\{ \Phi_{m,-n,l,-k}=\mathcal T_{(m T L/{\beta_1},\ell L \Omega)}\mathcal M_{(n\Omega, /{\beta_2}, kT)}\}_{m,n,k,\ell\in\Z}$ are tight Gabor frames with $A=\beta_2/T$, $A=\beta_1/\Omega$, and $A=\beta_1\beta_2/(T\Omega)=\beta_1\beta_2 L$, respectively, whenever $\beta_2\geq 1 +2\delta /T$ and $\beta_1\geq 1 +2\delta /\Omega$.
}
\end{theorem}

Observe that the reconstruction formulas given in Theorems~\ref{thm:reconstruction} and \ref{thm:trunc} require the functions $r$ and $\widehat\phi$ to generate partitions of unity \eqref{eq:r_phi_2a}, while \eqref{eq:r_phi_2} above requires that the functions obtained by taking the square of the modulus form  partitions of unity.  

\begin{theorem}\label{thm:localsamp}  Fix $M$ compact with $\mu(M)<1$, let $T,\Omega,L$ and $w,r,\phi$ be defined in Theorem~\ref{thm:reconstruction}, and let $(g,a\Z\times b\Z)$, $g\in \mathcal S(\R)$, be a tight frame for $L^2(\R)$ with frame bound 1. 

 Then $H\in OPW(M)$ with
  $ \|\sigma_H\|_{L^\infty(\R^2)}\leq \mu$,    and $\widetilde H\in OPW(M)$ defined via its symbol
\begin{equation*}
\widetilde\sigma (x,\xi)= \frac{LT}{\beta_1 \beta _2} \sum_{j=0}^{L-1} e^{-2\pi i(xn_j\Omega + \xi  k_j T )} e^{2\pi i n_j\Omega k_j T} \hspace{-.7cm}\sum_{(mLT/\beta_1,\ell L\Omega/\beta_2)\in S}\hspace{-.4cm} \widetilde\sigma^{(j)}_{m,\ell} \,V_\phi r(x - \big(\frac{m  L}{\beta_1}+k_j\big) T, \ \xi-\big(\frac{\ell L }{\beta_2}+n_j\big)\Omega\big), 
\end{equation*}
where
\begin{align*}
\widetilde \sigma^{(j)}_{m,\ell} = \sum_{q\in\Z}\  b_{jq}\ \phi \big((-q -k_j-mL/\beta_1)T\big)\ \langle H\widetilde\ident,\ \mathcal T_{qT}\mathcal M_{\ell \Omega L /\beta_2}\, r\rangle,
\end{align*}
satisfy
\begin{align*}
\|Hf - \widetilde H f\|_{L^2(\R)}\leq &\, C \, \varepsilon \, \mu\, \|f\|_{L^2(\R)}
\end{align*}
for all $f\in L^2(\R)$ which are $\epsilon$--time-frequency localized on $S-B(d(\epsilon))$ with respect to $(g,a\Z\times b\Z)$ in the sense of Definition~\ref{def:localization}.
Again $S\subseteq I_1\times I_2 \subset \R^2$ is given, $\varphi$ and $\widetilde w$ are defined as in Theorem~\ref{thm:trunc}, $C>0$, and $d$ can again be chosen independent of $S$ with $d(\epsilon)=o(\sqrt[k]{1/\epsilon})$ for all $k\in\N$ as $\epsilon\to 0$.

\end{theorem}

The discrete representations introduced in Theorems~\ref{prop:exp} and \ref{thm:localsamp} resolve a fundamental conceptual difference between classical sampling and operator sampling. In contrast to classical sampling, where the sampling values can be extracted individually, the contributions of the different Dirac-deltas in the operator sampling formula are combined in a single function and cannot easily be separated. Hence, while choosing a higher sampling rate in the function case yields more information,  in the operator case, this additional information is mixed in an inseparable way. These aliasing effects \cite{Pfander06-3} make it impossible to obtain redundant representations merely by oversampling in Theorem~\ref{thm:main-simple} or Theorem~\ref{thm:reconstruction}. In reconstruction formula~\eqref{eq:propexp}, however, the oversampling parameters $\beta_i$ can be chosen arbitrarily, allowing for representations of arbitrarily large redundancy.

This interplay of  large redundancy and  good local representation properties of the discrete coefficients can be exploited to coarsely quantize bandlimited operators, i.e., to represent these samples by values from a finite alphabet which allow for approximate recovery via the same reconstruction formulas as in Theorems~\ref{prop:exp} and \ref{thm:localsamp}. For such methods, as they have been studied in the mathematical literature for frame expansions over $\R^n$ \cite{BPY,BO,lapoyi08} or the space of bounded bandlimited  functions on $\R$ \cite{DaubDev03,Gunturk03exp,DGK10}, the possibility to oversample is of crucial importance. We will, however, leave this to future work.

\section{Local approximation of bandlimited operators}\label{sec:local}

In this section we show that a local approximation of an operator's symbol always yields a local approximation of the operator in the sense of Definition~\ref{def:localization}.
The given results are of general interest and will be stated in more general terms than  other results in this paper.  This does not increase the difficulty of proof, but necessitates to recall  additional terminology from time-frequency analysis.
%

For that, recall that for any full rank lattice $\Lambda={\rm A}\Z^{2d}\subseteq \R^{2d}$, $\det {\rm A}\neq 0$,  $\ell_s^{p}(\Lambda)$ denotes the set of sequences $(c_\lambda)_{\lambda\in\Lambda}$ for which
$$\|c\|_{\ell_s^{p}(\Lambda)} = \Big(\sum_{\lambda\in\Lambda}  \big| (\|\lambda\|^s_\infty+1)\, c_\lambda\big|^p \Big)^{1/p}<\infty.$$
A time-frequency shift by $\lambda=(t,\nu)\in\Lambda$ is denoted by $\pi(\lambda)=\mathcal M_\nu \mathcal T_t$ and in the following we will consider Gabor systems of the form $(g,\Lambda)=  \{\pi(\lambda) g  \}_{\lambda\in \Lambda}$.

Among the many equivalent definitions of modulation spaces, we choose the following.  Let $g_{0}(x)=e^{-\|x\|}$,  $1\leq p\leq \infty$ and $s\in\mathbb R$. Then
\begin{align}
 M_s^{p}(\R^d)=\{f\in \mathcal S'(\R^{d}):\  \|f\|_{M_s^{p}(\R^d)}= \| (\langle f,\pi(\lambda)g_{0}\rangle)_{\lambda}\|_{\ell_s^{p}(\frac 1 2 \Z^{2d}  )} <\infty\,\},  \label{eqn:modulationspace}
\end{align}
where we generally omit the subscript $s=0$. For details on modulations paces, see, for example, \cite{Groechenig01,Feichtinger89}. In the following we shall use the fact that whenever $(g,\Lambda)$ is a tight $L^{2}$-Gabor frame (see below for a precise definition) with $g\in M^{1}(\R^{d})$ then replacing the $L^{2}$-Gabor frame $(g_{0},\frac 1 2 \Z^{2d})$  in \eqref{eqn:modulationspace} with  $(g,\Lambda)$ leads to an equivalent norm on $M^{p}(\R^d)$ \cite{Groechenig01}. That is, there exist positive constants $A$ and $B$ with 
\begin{equation}\label{eqn:mpFrame}
 A\|f\|_{M^{p}(\R^{d})}^p\leq \sum_{\lambda\in\Lambda} |\langle f,\pi(\lambda)g \rangle|^p \leq B\|f\|_{M^{p}(\R^{d})}^p,\quad f\in M^p(\R^d)
\end{equation}
if $1\leq p< \infty$ and
$$ A\|f\|_{M^{\infty}(\R^{d})} \leq \sup_{\lambda\in\Lambda} |\langle f,\pi(\lambda)g \rangle| \leq B\|f\|_{M^{\infty}(\R^{d})},\quad f\in M^\infty(\R^d)$$
 if $p=\infty$. 
In either case, we call $(g,\Lambda)$ an $\ell^p$-frame with lower frame bound A and upper frame bound B. If we can choose $A=B$ in case of $p=2$ then we call $(g,\Lambda)$ a tight Gabor frame.


The norm equivalence \eqref{eqn:equivalentnorms} follows from the following result since $M^2(\R)=L^2(\R)$.

\begin{theorem} \label{thm:normequiv} Let $1\leq p \leq \infty$ and $M$ compact. Then there exist positiv constants $A=A(M,p)$ and $B=B(M,p)$ with
\begin{align*}
 A \, \|\sigma_H\|_{L^\infty(\R^2)}\leq \|H\|_{\mathcal L(M^{p}(\R))}\leq B\, \|\sigma_H\|_{L^\infty(\R^2)}, \quad H\in OPW (M).
\end{align*}
\end{theorem}
\begin{proof}
 Theorem~2.7 in \cite{Pfa10} (see for example the proof of Theorem~3.3 in \cite{Pfa10}) provides $C=C(M,p)$ with
\begin{align*}
  \|Hf\|_{M^{p}(\R)}\leq C\, \|\sigma_H\|_{L^\infty(\R^2)}\, \|f\|_{M^{p}(\R)} 
\end{align*}
for all $H\in OPW (M)$. This establishes the existence of $B=B(M,p)$ above.

In addition, we shall use the following facts. 
In \cite{FG92,Groechenig01} it is shown that the operator norm of an operator mapping the modulation space $M^1(\R)$ into its dual $M^\infty(\R)$ is equivalent to the $M^\infty(\R^2)$ norm of its kernel $\kappa$, which can easily shown to be equivalent to the $M^\infty(\R^2)$ norm of the time-varying impulse response $h$. Moreover, we use the fact that $M^\infty(\R^2)$ is invariant under Fourier transforms (in some or all variables) and that the $M^\infty(\R^2)$ norm can be replaced by the $L^\infty(\R^2)$ norm if we restrict ourselves to functions bandlimited to a fixed  set $M$ \cite{Oko09,Pfa10}. Last but not least, we use that the identity map embedding $M^p(\R)$ into $M^q(\R)$, $p\leq q$, is bounded.

Writing $\lesssim$ to express that $A\leq C B$ for some constant $C$ depending only on the support $M$ and $A\asymp B$ to denote equivalence in norms, i.e., $A\lesssim B$ and $B\lesssim A$, we obtain for all $H\in OPW(M)$
\begin{align*} 
 \|\sigma_H\|_{L^\infty(\R^2)} 
\asymp \|\sigma_H\|_{M^\infty(\R^2)}
 \asymp \|h_H\|_{M^\infty(\R^2)}
\asymp \|\kappa_H\|_{M^\infty(\R^2)}
 \asymp \|H\|_{\mathcal L (M^1(\R),M^\infty(\R))}  
\lesssim  \|H\|_{\mathcal L (M^p(\R))}
\end{align*}
and the result follows.
\end{proof}

We proceed to prove the following generalization of Theorem~\ref{thm:localization1}. Indeed, the earlier stated result follows again from the fact that $L^2(\R)=M^2(\R)$ and $g \in \mathcal S(\R)$ implies $g\in M_s^1(\R)$ for all $s\geq1$. We focus on the case of arbitrary domains; a simpler proof for rectangular domains can be obtained using  Theorem~\ref{thm:main-simple} instead of Theorem~\ref{thm:reconstruction}.

\begin{theorem}\label{prop:exp-general}
Fix $M$ compact and $p\in[1,\infty]$. Let $(g,\Lambda)$, $g\in M_s^1(\R)$, $s \geq 1$, be a tight frame for $L^2(\R)$ with frame bound $1$.
Then any $H\in OPW(M)$ with
  $$
    \|\sigma_H\|_{L^\infty(\R^2)}\leq \mu\quad\text{and}\quad \|\sigma_H\|_{L^\infty(S)}\leq \epsilon\,\mu,
  $$
  satisfies
  \begin{align*}
    \|Hf\|_{M^p(\R)}\leq 
    C\,\epsilon\,\mu\,  \|f\|_{M^p(\R)}
  \end{align*}
  for all $f\in M^p(\R)$  time-frequency localized on  $S-B\big(d(\epsilon)\big)=\Big(S^c+B\big(d(\epsilon)\big)\Big)^c$ in the sense that, for $p<\infty$,
  \begin{align*}
        \sum_{ \lambda \in \Lambda \cap \normalsize(S-B\normalsize(d(\epsilon)\normalsize)\normalsize) } |\langle f,\pi(\lambda)g\rangle|^p\geq (1-\epsilon^p
        )\,\sum_{\lambda\in\Lambda} |\langle f, \pi(\lambda)g\rangle|^p \, ,\end{align*}
         or, for $p=\infty$, 
        \begin{align*}
        \sup \big\{ |\langle f,\pi(\lambda)g\rangle|,\  \lambda \in \Lambda \cap \big(S-B(d(\epsilon))\big) \}  \geq (1-\epsilon
        )\sup \big\{ |\langle f,\pi(\lambda)g\rangle|,\  \lambda \in \Lambda\big\} \,.
        \end{align*}
Here $C>0$ is an absolute constant and $d:(0,1)\longrightarrow \R^+$ is a function independent of $S$  which satisfies $ d(\epsilon)=o(\epsilon^{-1/s})$ as $\epsilon\to 0$.
\end{theorem}

\begin{proof}

\noindent  {\it Step 1. Preliminary observations and choice of auxiliary objects.}
\qquad
Choose a nonnegative $\phi\in \mathcal S(\R^2)$ with
$\int\phi(x)\,dx=1$ and $\supp \phi\subseteq{[-\frac 1 2, \frac 1 2]^2}$. Recall that 
$$
	\Lambda^{\perp}=\{ \mu\in\R^{2}:\ \ e^{2\pi i \langle \mu, \lambda\rangle}=1 \text{ for all } \lambda\in\Lambda \}
$$ 
is called dual  lattice of the lattice $\Lambda$ in $\R^{2}$.
Let $\widetilde \Lambda$ be a lattice containing $\Lambda$ with the property that there exists a compact and convex fundamental domain $D$ of $\widetilde \Lambda^\perp$ which contains $M+[-\frac 1 2, \frac 1 2]^2$.  Set
$$
    \sigma_P = \|\chi_{D}\ast \phi\|_{L^2(\R^2)}^{-1} \  \mathcal F (\chi_{D}\ast \phi)
$$
and, using the sampling theorem for lattices in $\R^n$ \cite{PM62,Groechenig01},  we obtain for all $H\in OPW(M)$
$$
   \sigma_H= \sum_{\lambda\in \widetilde\Lambda}\sigma_H(\lambda)\ \mathcal T_\lambda \sigma_P
$$
and hence
\begin{equation}
 H=\sum_{\lambda\in\widetilde\Lambda}    \sigma_H(\lambda)  \,  \pi(\lambda)P\pi(\lambda)^\ast. \label{eqn:Hdecomp}
\end{equation}

 As explained above, the fact that  $(g,\Lambda)$ is a Gabor frame in $L^2(\R)$ with $g\in M^1(\R)$, implies that it is also an
$\ell^p$-frame for $M^p(\R)$ and there exists  $C_1,C_2>0$ with
\begin{equation}\label{eqn:pframeproperty1}
     \|f\|_{M^p(\R)}\leq C_1\,\|\{\langle f, \pi(\lambda) g \rangle  \}_{ \lambda\in \Lambda}
    \|_{\ell^p(\Lambda)} \leq C_1C_2\,\|f\|_{M^p(\R)},\quad f\in M^p(\R).
\end{equation}
As the synthesis map is the adjoint of the analysis map, we also have 
\begin{equation}\label{eqn:pframeproperty2}
 \big\|\sum_{\lambda \in \Lambda} c_\lambda \pi(\lambda) g \big\|_{M^p(\R)}\leq C_2\,\|\{c_\lambda  \}_{ \lambda\in \Lambda}
    \|_{\ell^p(\Lambda)}.
    \end{equation}  
    
Since $\Lambda$ is a subgroup of $\widetilde \Lambda$, we have $\widetilde \Lambda=\bigcup_{\ell=1}^n (\Lambda+\mu_\ell)$ for some $\mu_1,\mu_2,\ldots,\mu_n$. Here $n$ is finite, as otherwise the set would be dense, hence not a discrete lattice, and depends only  on $M$ and $(g,\Lambda)$. It is easily seen that  $(g,\Lambda+\mu_\ell)$,  $\ell=1,\ldots,n$, also satisfies \eqref{eqn:pframeproperty1} and \eqref{eqn:pframeproperty2}.  Setting $\widetilde g=  n^{-1/2}\, g\in M_s^1(\R)$, we conclude that the Gabor system $(\widetilde g,\widetilde\Lambda)$ is a tight frame for $L^2(\R)$ with frame bounds equal 1 and an $\ell^p$-frame with for $M^p(\R)$ with
  \begin{align*}
    \|f\|_{M^p(\R)}&\leq C_1\, n^{\frac 1 2 - \frac 1 p}\,\|\{\langle f, \pi(\widetilde\lambda)\widetilde g \rangle  \}_{\widetilde \lambda\in \widetilde\Lambda}
    \|_{\ell^p(\widetilde \Lambda)}\\ &\leq C_1  \, n^{\frac 1 2 - \frac 1 p}\, C_2\, n^{\frac 1 p - \frac 1 2}\,\|f\|_{M^p(\R)}= C_1C_2\, \|f\|_{M^p(\R)},\quad f\in M^p(\R).
  \end{align*}
  
We claim that
  \begin{align}\label{eqn:ell1membership}
    \Big\{
        \langle P\pi(\lambda) \widetilde g, \pi(\widetilde\lambda) \widetilde g\rangle
    \Big\}\in \ell_s^1(\widetilde \Lambda\times\widetilde \Lambda)\,.
 \end{align}
To see this, recall that $\sigma_P\in \mathcal S(\R^2)\subseteq M_s^1(\R^2)$, and, hence, $\widetilde \sigma_P$ given by $\sigma_P(x,\xi)\, e^{2\pi i x \xi}$ is in $M^1_s(\R^2)$ as $e^{2\pi i x \xi}$ is a Fourier multiplier and hence also a time multiplier for $M_s^1(\R^2)$ (Lemma 2.1 in \cite{grhe99}, related results can be found in \cite{BGOR07, toft04, toft04b, toft07}). A direct computation implies that for $\lambda=(t,\nu)$ and $\widetilde \lambda=(\widetilde t,\widetilde \nu)$ we have
\begin{align*}
   |\langle P\pi(t,\nu)\widetilde g, \pi(\widetilde t,\widetilde \nu) \widetilde g \rangle |
    &= \Big|\iint \sigma_P(x,\xi)\, e^{2\pi i x \xi}\, \widehat{\mathcal M_\nu \mathcal T_t \widetilde g}(\xi) \overline{\mathcal M_{\widetilde \nu} \mathcal T_{\widetilde t} \widetilde g(x)}\, d\xi\,dx\Big|\\
    &= \Big|\iint \sigma_P(x,\xi)\, e^{2\pi i x \xi}\, \mathcal M_{-t} \mathcal T_\nu \widehat{\widetilde g}(\xi) \overline{\mathcal M_{\widetilde \nu} \mathcal T_{\widetilde t} \widetilde g(x)}\, d\xi\,dx\Big|\\
    &= \big|\langle  \widetilde{\sigma}_P , \mathcal M_{(\widetilde \nu, t)} \mathcal T_{(\widetilde t, \nu)} \widetilde g{\otimes} \overline{\widehat{\widetilde g} } \rangle \big|.
\end{align*}
Equation \eqref{eqn:mpFrame} implies that the right hand side defines an $\ell_s^1(\widetilde \Lambda \times \widetilde \Lambda)$ sequence since $\widetilde{\sigma_P}\in  M_s^{1}(\R^{2})$ and $(g{\otimes}\overline{ \widehat {\widetilde g}},\, \widetilde \Lambda\times  \widetilde \Lambda)$ is a Gabor frame with window $\widetilde g{\otimes}\overline{ \widehat {\widetilde g}}$ in $M_s^{1}(\R^{2})$. Hence, \eqref{eqn:ell1membership} holds and   for $\{S_k\}_{k\in\N_0}$ defined by 
 $$S_k=
  \sum_{\|(\lambda,\widetilde\lambda)\|_\infty=k} |
        \langle P\pi(\lambda) \widetilde g, \pi(\widetilde\lambda) \widetilde g\rangle|,
   $$
we have $ \{ (k+1)^s S_k\} \in \ell^1(\N)$. That is, $\{S_k\}=o(k^{-(s+1)})$ and for some $C>0$ we have $\sum_{k=K}^\infty S_k \leq  C\ K^{-s}$, $K\in\N$.

For $\epsilon>0$ set 
 $d(\epsilon)= (C/ \epsilon)^{1/s} $ and observe that then 
$$ 
    \sum_{\widetilde \lambda \in\widetilde\Lambda}\ \sum_{\lambda\in\widetilde\Lambda\cap B(d(\epsilon))^c}   \Big| \langle P\pi(\lambda)\widetilde g, \pi(\widetilde\lambda)\widetilde g\rangle\Big| \leq   \sum_{k= d(\epsilon)}^\infty S_k  \leq C \big( (C/\epsilon)^{1/s}\big)^{-s}= \epsilon.
$$
Now, set $A(\lambda,\widetilde \lambda)=\langle P\pi(\lambda) \widetilde g, \pi(\widetilde\lambda)\widetilde g\rangle$ if   $\lambda \in \widetilde \Lambda\cap B(d(\epsilon))^c$ and $0$ else.

\vspace{.5cm}
\noindent  {\it Step 2. Decomposing $Hf$ as $Hf=H_{\rm in}f_{\rm in}+H_{\rm out}f_{\rm in}+Hf_{\rm out}$.}\qquad We set
$$\Lambda_{\rm in}=\Lambda\cap \big(S-B(d(\epsilon))\big),\quad \widetilde\Lambda_{\rm in}=\widetilde \Lambda\cap  \big(S-B(d(\epsilon))\big),
\quad \Lambda_{\rm out}=\Lambda \setminus \Lambda_{\rm in}, \quad \widetilde{\Lambda}_{\rm out}=\widetilde \Lambda \setminus \widetilde{\Lambda}_{\rm in}.$$
and
$$f_{\rm in}=\sum_{\lambda\in \Lambda_{\rm in}} \langle f, \pi(\lambda) g\rangle\,  \pi(\lambda) g
=\sum_{\lambda\in \widetilde{\Lambda}_{\rm in}} c_\lambda\,  \pi(\lambda)\widetilde g, \quad f_{\rm out}=f-f_{\rm in},$$
where $c_\lambda= \sqrt{n}\, \langle f, \pi(\lambda) g\rangle$ if $\lambda\in\Lambda$ and $0$ else.
Similarly, inspired by \eqref{eqn:Hdecomp}, we set for  $H\in OPW(M)$
$$ H_{\rm in}=\sum_{\lambda\in\widetilde \Lambda \cap S} \sigma_H(\lambda) \, \pi(\lambda)P\pi(\lambda)^\ast, \quad
H_{\rm out}=H-H_{\rm in}\,
$$
and note that $H_{\rm in},H_{\rm out}\in OPW(D+[-\tfrac 1 2, \tfrac 1 2]^2)$.

\vspace{.5cm}
\noindent {\it Step 3. Bounding $\|H_{\rm out} f_{\rm in}\|_{M^{p}(\R)}$.}\qquad
We use the separation of $\widetilde{\Lambda}_{\rm in}$ and $\widetilde{\Lambda}\cap S^c$ by $d(\epsilon)$ to compute
\begin{align*}
  |\langle H_{\rm out}f_{\rm in}, \pi(\widetilde\lambda)\widetilde g \rangle|
    &=\Big|\langle \sum_{\nu\in  \widetilde{\Lambda}\cap S^c}\sigma_H(\nu)  \pi(\nu) P \pi(\nu)^\ast\sum_{\lambda\in \widetilde{\Lambda}_{\rm in}}c_\lambda  \pi(\lambda) \widetilde g ,  \,  \pi(\widetilde \lambda )\widetilde g\rangle  \Big| \\
   &\leq \sum_{\nu\in  \widetilde{\Lambda}\cap S^c}|\sigma_H(\nu)|\,\sum_{\lambda\in \widetilde{\Lambda}_{\rm in}}|c_\lambda|\, \Big|\langle  \pi(\nu) P \pi(\nu)^\ast \pi(\lambda)  \widetilde g, \, \pi(\widetilde\lambda) \widetilde g \rangle \Big| \\
    &\leq \sum_{\nu\in  \widetilde{\Lambda}\cap S^c}|\sigma_H(\nu)|\,\sum_{\lambda\in \widetilde{\Lambda}_{\rm in}}|c_\lambda|\, \Big|\langle  P \pi(\lambda-\nu) \widetilde  g, \, \pi(\widetilde\lambda-\nu) \widetilde g \rangle \Big|\\
    &\leq \sum_{\nu\in  \widetilde{\Lambda}\cap S^c}|\sigma_H(\nu)|\,\sum_{\lambda\in \widetilde{\Lambda}_{\rm in}}|c_\lambda|\, A(\lambda-\nu,\widetilde\lambda-\nu)\\
     &\leq \|\sigma_H\|_{L^\infty(\R^2)}\sum_{\nu\in \widetilde\Lambda}\,\sum_{\lambda\in \widetilde \Lambda}|c_\lambda|\, A(\lambda-\nu,\widetilde\lambda-\nu).
\end{align*}
For every sequence $\{ d_\lambda \}\in \ell^q(\widetilde \Lambda)$, $1/p+1/q=1$,  we conclude
\begin{align*}
 \big|\langle \{\langle H_{\rm out}f_{\rm in}, \pi(\widetilde\lambda)\widetilde g \rangle \}_{\widetilde \lambda\in \widetilde \Lambda}
   &,\{ d_{\widetilde \lambda}\}_{\widetilde \lambda\in \widetilde \Lambda} \rangle \big| \\ &\leq \|\sigma_H\|_{L^\infty(\R^2)}
    \sum_{\widetilde \lambda\in \widetilde \Lambda}\sum_{\nu\in \widetilde \Lambda}\,\sum_{\lambda\in \widetilde \Lambda}|c_\lambda|\, A(\lambda-\nu,\widetilde\lambda-\nu)\, |d_{\widetilde \lambda}|\\
    &=\|\sigma_H\|_{L^\infty(\R^2)}
    \sum_{ \lambda\in \widetilde \Lambda}\sum_{\widetilde\lambda\in \widetilde \Lambda} \sum_{\nu\in \widetilde \Lambda}\,|c_{\lambda+\nu}|\, A(\lambda,\widetilde \lambda )\, |d_{\widetilde \lambda+\nu}|\\
    &\leq\|\sigma_H\|_{L^\infty(\R^2)}\|\{c_\lambda\}\|_{\ell^p(\widetilde \Lambda)}\|\{d_\lambda\}\|_{\ell^q(\widetilde \Lambda)}
    \sum_{ \lambda\in \widetilde \Lambda}\sum_{\widetilde\lambda\in \widetilde \Lambda}  A(\lambda,\widetilde \lambda )\,
\end{align*}
and
 \begin{align}
  \notag
\|H_{\rm out}f_{\rm in}\|_{M^p(\R)}&\leq  n^{\frac 1 2 - \frac 1 p}\, C_1  \|\{|\langle H_{\rm out}f_{\rm in}, \pi(\widetilde\lambda)\widetilde g \rangle|\} \|_{\ell^p(\widetilde \Lambda)}\\ \notag
    &\leq n^{\frac 1 2 - \frac 1 p}\, C_1\|\sigma_H\|_{L^\infty(\R^2)}\|\{c_\lambda\}\|_{\ell^p(\widetilde \Lambda)}\sum_{ \lambda\in \widetilde \Lambda}\sum_{\widetilde\lambda\in \widetilde \Lambda}  A(\lambda,\widetilde \lambda )\notag
    \\
      &\leq n^{\frac 1 2 - \frac 1 p}\, C_1 \mu\, \|\{ n^{\frac 1 2}\, \langle f, \pi(\lambda)g\rangle \}\|_{\ell^p( \Lambda)} \,\epsilon \notag
  \notag \leq n^{1 - \frac 1 p}\, C_1\, C_2\,\epsilon  \,\mu\, \|f\|_{M^p(\R)}   .
\end{align}

\vspace{.5cm}
\noindent {\it Step 4. Bounding $\|H f_{\rm out}\|_{M^p(\R)}$.} \qquad
By Proposition~\ref{thm:normequiv} we have
$$
\|Hf_{\rm out}\|_{M^p(\R)}\leq B(M,p)\, \|\sigma_H\|_{L^\infty(\R^2)}\|f_{\rm out}\|_{M^p(\R)}.
$$
By hypothesis, for $p<\infty$ we have
\begin{align*}
 \|f_{\rm out}\|_{M^p(\R)}^p &= \|\sum_{\lambda\in\Lambda_{\rm out}} \langle f,\pi( \lambda) g\rangle\pi( \lambda) g\|_{M^p(\R)}^p 
    \leq C_2^p \sum_{\lambda\in\Lambda_{\rm out}}| \langle f,\pi( \lambda) g\rangle|^p \\
    &\leq C_2^p \, \epsilon^p \sum_{\lambda\in\Lambda}| \langle f,\pi( \lambda) g\rangle|^p 
    \qquad \quad\ \,\leq C_2^{2p} \, \epsilon^p \|f\|_{M^p(\R)}^p\,,
\end{align*}
and 
for $p=\infty$ we have
\begin{align*}
 \|f_{\rm out}\|_{M^\infty(\R)} &= \|\sum_{\lambda\in\Lambda_{\rm out}} \langle f,\pi( \lambda) g\rangle\pi( \lambda) g\|_{M^\infty(\R)} 
    \leq C_2\, \| \langle f,\pi( \lambda) g\rangle\|_{\ell^\infty(\Lambda_{\rm out})} \\
    &\leq C_2\, \epsilon \,\| \langle f,\pi( \lambda) g\rangle\|_{\ell^\infty(\Lambda)} 
    \qquad \quad\ \ \,\leq C_2^2\, \epsilon \,\|f\|_{M^\infty(\R)}\,.
\end{align*}
We conclude
$$
\|Hf_{\rm out}\|_{M^p(\R)}\leq B(M,p)\,C_2\,\epsilon\, \ \|\sigma_H\|_{L^\infty}\|f\|_{M^p(\R)}\leq B(M,p)\,C_2\,\epsilon\,  \mu\, \|f\|_{M^p(\R)}.
$$

\vspace{.5cm}
\noindent {\it Step 5. Bounding $\|H_{\rm in} f_{\rm in}\|_{M^p(\R)}$.}\qquad  Since $\sigma_P\in \mathcal S(\R^2)$, the operator 
$$ \ell^\infty(\Lambda)\rightarrow L^\infty(\R^2),\quad \{c_\lambda\} \mapsto 
\sum_{\lambda\in \widetilde\Lambda}c_\lambda\ \mathcal {\mathcal T}_\lambda \sigma_P$$ is bounded, say with operator norm bound $C_3$.
Then, Proposition~\ref{thm:normequiv} implies
\begin{align*}
  \|H_{\rm in} f_{\rm in} \|_{M^p(\R)}&\leq B(D{+}[-\tfrac 1 2, \tfrac 1 2]^2,p)\,  \|\sigma_{H_{\rm in}}\|_{L^\infty(\R^2)}\|f_{\rm in}\|_{M^p(\R)}\\
  &\leq  B(D{+}[-\tfrac 1 2, \tfrac 1 2]^2,p)\,  C_3\,  \|\{\sigma_{H}(\lambda)\} \|_{\ell^\infty(\widetilde \Lambda \cap S)} (1+\epsilon) \|f\|_{M^p(\R)} \\ &\leq  2\,B(D{+}[-\tfrac 1 2, \tfrac 1 2]^2,p)\, C_3\, \epsilon \, \mu\,   \|f\|_{M^p(\R)}.
\end{align*}

Since all constants are independent of $\epsilon$, $\mu$, $H$, and $f$, we summarize
\begin{align*}
\hspace{2.5cm}  \|Hf\|_{M^{p}(\R)}=\|H_{\rm in}f_{\rm in}+H_{\rm out}f_{\rm in}+Hf_{\rm out}\|_{M^p(\R)}\leq  C \epsilon\, \mu\,\|f\|_{M^{p}(\R)}\,. \hspace{2.5cm} \qedhere
\end{align*}
\end{proof}

\section{Operator identification using localized identifiers}\label{section:identifier}
This section analyzes identifiers that are localized in time and frequency. Theorem~\ref{thm:nodecay} shows that such functions cannot serve as an identifier for the entire Paley-Wiener.

{\it Proof of Theorem~\ref{thm:nodecay}.} \ 
Let $r\neq 0$ be a Schwartz function with $\supp r\subseteq [0,T]$ and $\phi \neq 0$ be a Schwartz function with $\supp \widehat \phi \subseteq [-\Omega/2,\Omega/2]$. Let $H_n$ be defined via its kernel $\kappa_n(x,y)=\phi(x-n)r(x-y)$, so $h_n(x,t)=\phi(x-n)r(t)$ and $\eta_n(t,\nu)=\int h_n(x,t)e^{-2\pi i x \nu}dx=r(t)e^{2\pi i n \nu}\widehat \phi(\nu)$, so $H_n\in OPW([0,T]{\times} [-\Omega/2,\Omega/2])$ with  
$\|\sigma_{H_n}\|_{L^\infty(\R^2)}= \|\widehat r\|_{L^\infty(\R)}\, \|\phi \|_{L^\infty(\R)}$. 

If $w$ identifies  $OPW([0,T]{\times} [-\Omega/2,\Omega/2])$, then 
by definition $H_nw\in L^2(\R)$. Then
\begin{align*}
 \int | H_n w(x)|^2\, dx  &=  \int |\langle \kappa_n(x,y), \ w(y) \rangle_{y}|^2\, dx
=  \int |  \phi(x-n) |^2 \,| \langle r(x-y), \, w(y)\,\rangle_{y} |^2\, dx.
\end{align*}
Clearly, $\langle r(x-y), \, w(y)\,\rangle_{y} \stackrel{x\to \pm \infty}{\longrightarrow} 0$ would imply 
$ \| H_n w\|_{L^2(\R)} \stackrel{n\to \pm \infty}{\longrightarrow} 0$ and contradict identifiability \eqref{eqn:identifable} since
by \eqref{eqn:equivalentnorms} we have
$ \|H_n\|_{\mathcal L(L^2(\R))}\geq A\|\sigma_{H_n}\|_{L^\infty(\R^2)}= A\|\widehat r\|_{L^\infty(\R)}\, \|\phi \|_{L^\infty(\R)}$ for all $n\in \Z$.

To show that an identifier $w$ cannot decay in frequency, we choose
$H_n\in OPW([0,T]{\times} [-\frac{\Omega}{2},\frac{\Omega}{2}])$ to have  spreading functions $\eta_n(t,\nu) =r(t)e^{2\pi i n t}\widehat\phi(\nu)e^{-2\pi i t \nu}$.  Let  $g$ be a Schwartz function and compute using Fubini's Theorem and, for notational simplicity, using bilinear pairings in place of sesquilinear ones, 
\begin{align*}
 \langle H_n w(x), g(x) \rangle_x   &=  \big\langle  \eta_n(t,\nu), \langle e^{2\pi i x \nu } w(x-t),\, g(x) \rangle_x \big\rangle_{t,\nu}\\
&=  \big\langle  r(t) e^{2\pi i n t}\widehat \phi(\nu), \langle e^{2\pi i (x-t) \nu } w(x-t),\, g(x) \rangle_x \big\rangle_{t,\nu}\\
&=  \big\langle  r(t) e^{2\pi i n t}\, w(x-t)\, g(x),\, \langle \widehat \phi(\nu), e^{2\pi i (x-t) \nu } \rangle_\nu   \big\rangle_{t,x}\\
&=  \big\langle \langle   r(t) e^{2\pi i n t}, \,   w(x-t)\,  \phi(x-t)\rangle_{t},\, g(x)\big\rangle_{x}\\
&=  \big\langle \langle  \widehat r(\xi -n), \, e^{-2\pi i x \xi }\  \widehat w\ast \widehat  \phi(\xi)\rangle_{\xi},\, g(x)\big\rangle_{x}
=  \big\langle   \widehat r(\xi -n)\,   \widehat w\ast \widehat  \phi(\xi),\, \widehat g(\xi)\big\rangle_{\xi}\,.
\end{align*}
Hence,
$$
\|H_nw\|^2_{L^2(\R)}
	=\|\widehat{H_n w}\|^2_{L^2(\R)} 
	= \int| \widehat r(\xi -n)|^2\,   |\langle \widehat w(\xi-\nu) ,\, \widehat  \phi(\nu)\rangle_\nu|^2\,d\xi \,,
$$
and we can conclude as above. \qed

We proceed by showing that local identification of operators is possible with identifiers localized both in time and frequency, Theorem~\ref{thm:trunc}.

{\it Proof of Theorem~\ref{thm:trunc}.} 
The proof proceeds in two steps. First we show that replacing each Dirac-delta by a suitable smoothed out version locally introduces only a small error and identification using the resulting smooth identifier can be interpreted as sampling a modified bandlimited operator. 
Second we show that reducing to a finite number of samples also locally yields only a small error. Applying this to the modified operator arising in the first part proves that both reductions together also yield only a small error.

For the first part, choose $\varphi \in \mathcal S$ with $\supp \varphi\subseteq [-\delta, \delta]$, $\|\widehat \varphi\|_{L^\infty(\R)}=1$, and   $ |\widehat \varphi(\xi)-1|\leq \epsilon$ for $\xi\in I_2$.  Define $C_\varphi: f\mapsto f\ast \varphi$ and set $H_C=H\circ C_\varphi$. Observe that
\begin{align*}
 	H_Cf(x)&=\int \!\!\!\int \eta_H(t,\nu)e^{2\pi i x \nu} f\ast\varphi(x-t)\, dt\,d\nu
		\\&=\int \!\!\!\int \!\!\!\int \eta_H(t,\nu)e^{2\pi i x \nu} f(x-t-y) \varphi( y)\, dy\, dt\,d\nu
		\\&=\int \!\!\!\int   \!\!\!\int \eta_H(t-y,\nu)e^{2\pi i x \nu} f(x-t) \varphi( y)\, dy\, dt\,d\nu
		\\&=\int \!\!\!\int \Big(  \int \eta_H(t-y,\nu) \varphi( y)\, dy\Big)\, e^{2\pi i x \nu} f(x-t)dt\,d\nu\,,
\end{align*}
that is,  $\eta_{H_C}(t,\nu)=\eta_H(\cdot,\nu)\ast \varphi( t)$ and 
$$\supp \eta_{H_C}\subseteq \supp \eta_{H}+[-\delta,\delta]{\times}\{0\}.$$
We can apply Theorem~\ref{thm:reconstruction} for the operator $H_C$ with $M_1:=M+[-\delta,\delta]{\times}\{0\}$ in place of $M$. As by assumption $M_1+[-\delta, \delta]^2$ still has measure less than one, this can be done with $\delta$, $r$ and $\phi$ as given in the theorem. Defining $\ident_1:=\varphi*\ident$, we obtain
\begin{equation}
 \kappa_ {H_C}(x+t,x) = LT  \sum_{j=0}^{L-1} r(t-k_jT)\Big( \sum_{q\in\Z} b_{jq} H\ident_1(t- (k_j-q)T) \phi(x+(k_j-q)T)\Big)\, e^{2\pi i n_j \Omega x}.\label{eqn:reconstructionformulaSmooth}
\end{equation}

Observe that
\begin{align*}
\sigma_{H_C}(x,\xi)=\mathcal F_s \eta_{H_C}(x,\xi)=\sigma_H(x,\xi)\,\widehat \varphi( \xi),
\end{align*}
and, by hypothesis, we have $\|\sigma_{H_C}\|_{L^\infty(\R^2)}\leq \|\sigma_H\|_{L^\infty(\R^2)} \leq \mu$ and $\|\sigma_H-\sigma_{H_C}\|_{L^\infty(S)}\leq \epsilon \mu$.

Note that for $I_1=\R$, \eqref{eqn:reconstructionformulaSmooth} agrees with \eqref{eqn:reconstructionformulaTilde} and we have $H_C=\widetilde H$, so this establishes the result.

For the second part, let us assume $S\subseteq I_1 \times \R$ and $M_1\subset [c,d]\times\R$.  Let $\psi\in \mathcal S(\R)$ be nonnegative and  satisfy $\sum_{n} \psi(x-nT)=1$ and $\supp \widehat \psi\subset [-1/T,1/T]$. Such a function can be obtained by choosing an arbitrary bandlimited, nonnegative $\psi_0\in\S$  with $\|\psi_0\|_{L^1}=1$ and defining $\psi=\chi_{[0,T]}*\psi_0$.

Set $P_A(x)=\sum_{nT\in A}\psi (x-nT)$, so $P_{[-N,N]}\to 1$ and $P_{[-N,N]^c}\to 0$   uniformly on compact subsets as $N\to \infty$. Moreover, $|P_A(x)|\leq 1$ for all $A$. Choose $N(\epsilon)$ so that $|\, P_{I_1+[-N(\epsilon),N(\epsilon)]}(x)-1|\leq \epsilon$ for $x\in I_1+[c,d]$ and choose $R(\epsilon)$ with
\begin{align}
 \sum_{{qT\notin I_1+[-R(\epsilon),R(\epsilon)] }}   \| P_{[I_1+[-N(\epsilon),N(\epsilon)]}(x)\,  V_{\phi^\ast}r (x-q,\xi) \|_{L^1(\R^2)} < \epsilon (1-\epsilon) D. \label{eqn:D}
 \end{align}
 where the nature of   $D$ is derived  by the computations below. The existence of such $R(\epsilon)$ follows from the fact that  $P_{I_1+[-N(\epsilon),N(\epsilon)]}(x)$ and $\,  V_{\phi^\ast}r$ decay faster than any polynomial. Furthermore, as $\psi\in {\mathcal{S}}$, a similar argument to the one given in the proof of \ref{prop:exp-general} shows that for both $R(\epsilon)$ and $N(\epsilon)$, the growth rate is again bounded by $o(\sqrt[k]{1/\epsilon}$ for arbitrarily large $k\N$.

Let $w_2 =\sum_{kT\in  I_1+[-R(\epsilon),R(\epsilon)]+[-\delta, T+\delta]}c_k\delta_{kT}$ and observe that $\widetilde H$ as defined in the theorem satisfies 
\begin{align*}
 h_{\widetilde H}(x+t,t)&=\kappa_{\widetilde H}(x+t,x) \\
&= LT  \sum_{j=0}^{L-1} r(t-k_jT)\Big( \sum_{q\in\Z} b_{jq} H_C\ident_2(t- (k_j-q)T) \phi(x+(k_j-q)T)\Big)\, e^{2\pi i n_j \Omega x}. 
\label{eqn:reconstructionformulaTilde}
\end{align*}

Since $M_1\subset [c,d]\times\R$, we have $\supp H_C\delta_y \subseteq [c+y,d+y]$, and therefore, 
\begin{align*}
  H_Cw(x)=H_C\sum_{k\in\Z}c_k \delta_{kT}(x) \ = \ H_C\!\!\sum_{kT\in I_1+[-R(\epsilon),R(\epsilon)]+[-T-\delta, \delta]+[c,d]} \!\!& c_k \delta_{kT}(x) = H_C w_2(x),
  \\  x \in K\equiv  &I_1+[-R(\epsilon),R(\epsilon)]+[-T-\delta, \delta] \,.
\end{align*}

Note that $\widetilde H \in OPW(M_2)$, where $M_2=M_1+[-\delta,\delta]^2$ (for details, see, for example, \cite{PW12}).
As $M_2+[-\delta,\delta]^2$ still has measure less than one, this implies that we can apply Theorem~\ref{thm:reconstruction} again with the same $\delta$. We obtain 
\begin{align*}
 h_{ H_C}&(x+t,t)- h_{\widetilde H}(x+t,t)\\
=& LT  \sum_{j=0}^{L-1} r(t-k_jT)\Big( \sum_{q\in\Z} b_{jq} H_C\big(\ident -\ident_2\big)(t- (k_j-q)T) \, \phi(x+(k_j-q)T)\Big)\, e^{2\pi i n_j \Omega x} 
 \\
 =& LT  \sum_{j=0}^{L-1} r(t-k_jT)\Big( \!\!\sum_{qT\notin K -(t -k_j T) } \!\!b_{jq} H_C\big(\ident -\ident_2\big)(t- (k_j-q)T) \, \phi(x+(k_j-q)T)\Big)\, e^{2\pi i n_j \Omega x} 
 \\
 =& LT  \sum_{j=0}^{L-1} r(t-k_jT)\Big( \!\!\!\sum_{qT\notin I_1+[-R(\epsilon),R(\epsilon)] } \!\!\!b_{jq} H_C\big(\ident -\ident_2\big)(t- (k_j-q)T) \, \phi(x+(k_j-q)T)\Big)\, e^{2\pi i n_j \Omega x}. 
\end{align*}

Setting $\widetilde K = K^c +[-\delta,T+\delta]$  
and using that $(\sigma_{H_C}(x,\xi)-\sigma_{\widetilde H}(x,\xi))\, P_{I_1+[-N(\epsilon),N(\epsilon)]}(x)$ is bandlimited to $M+\{0\}{\times}[-1/T,1/T])$, we  compute
\begin{align*}
 \|\sigma_{H_C}&-\sigma_{\widetilde H}\|_{L^\infty(S)} 
 	\leq 
	1/(1-\epsilon)\, \|(\sigma_{H_C}(x,\xi)-\sigma_{\widetilde H}(x,\xi))\, P_{I_1+[-N(\epsilon),N(\epsilon)]}(x)\|_{L^\infty(\R^2)}
	\\ & \asymp 
	1/(1-\epsilon)\, 
		\|(\sigma_{H_C}(x,\xi)-\sigma_{\widetilde H}(x,\xi))\, P_{I_1+[-N(\epsilon),N(\epsilon)]}(x)\|_{M^\infty(\R^2)}
	\\ & \asymp 
		1/(1-\epsilon)\, \|(h_{H_C}(x,t)-h_{\widetilde H}(x,t))\, P_{I_1+[-N(\epsilon),N(\epsilon)]}(x)\|_{M^\infty(\R^2)}
	\\ & \asymp 
		1/(1-\epsilon)\, 
		\Big\|LT \ P_{  I_1+[-N(\epsilon),N(\epsilon)]}(x) \sum_{j=0}^{L-1} r(t-k_jT)\, e^{2\pi i n_j \Omega (x-t)}\, 
		\\& \hspace{1.5cm} \sum_{qT\notin I_1+[-R(\epsilon),R(\epsilon)]  } b_{jq} H_C(w- w_2)(t- (k_j-q)T) \, 
				\phi(x-t+(k_j-q)T) \Big\|_{M^\infty(\R^2)}
\\ & \leq 
		LT/(1-\epsilon)\, 
		\sum_{j=0}^{L-1}  \sum_{qT\notin I_1+[-R(\epsilon),R(\epsilon)]  } \Big\| \ P_{I_1+[-N(\epsilon),N(\epsilon)]}(x) \, r(t-k_jT)\, e^{2\pi i n_j \Omega (x-t)}\, 
		\\& \hspace{3.5cm} b_{jq} {H_C}(w- w_2)(t- (k_j-q)T) \, 
				\phi(x-t+(k_j-q)T) \Big\|_{M^\infty(\R^2)}
\\ & \leq 
		LT/(1-\epsilon)\, 
		\sum_{j=0}^{L-1}  \sum_{qT\notin I_1+[-R(\epsilon),R(\epsilon)]  }\Big\| {H_C}(w- w_2)(t- (k_j-q)T) \Big\|_{M^\infty(\R^2)}
\\ 
&\hspace{2cm} \Big\| \ P_{I_1+[-N(\epsilon),N(\epsilon)]}(x) \, r(t-k_jT)\, e^{2\pi i n_j \Omega (x-t)}\, 
		 b_{jq}  \, 
				\phi(x-t+(k_j-q)T) \Big\|_{M^1(\R^2)}\,   \\ & \leq 
		\| H_C\|_{{\mathcal L}(M^\infty(\R))} \, \|w- w_2\|_{M^\infty(\R)} \frac{LT}{1-\epsilon}\, 
		\\ 
&\hspace{2cm}\sum_{j=0}^{L-1}  \sum_{qT\notin I_1+[-R(\epsilon),R(\epsilon)] } |   b_{jq}   |\,  \Big\| P_{I_1+[-N(\epsilon),N(\epsilon)]}(x) \,   r(t)\, 
				\phi(x-t-qT) \Big\|_{M^1(\R^2)},
 \end{align*}
 where we used the invariance of the $M^\infty$ and $M^1$ norm under translation and modulation and, for the last inequality, Theorem~\ref{thm:normequiv} -- noting that, for functions constant in one of the coordinate directions, the $M^\infty(\R)$ and $M^{\infty}(\R^2)$ norms agree. 
The second to last inequality is based on $M^1(\R^2)$ being a Banach algebra, namely on  $\|g_1 g_2\|_{M^1(\R^2)}\leq \|g_1\|_{M^1(\R^2)}\|g_2\|_{M^1(\R^2)}$ for $g_1,g_2\in M^1(\R^2)$.  Indeed,  for $f\in M^\infty(\R^2)$ and $g\in M^1(\R^2)$, we have 
\begin{align*}
\|fg\|_{M^\infty(\R^2)} 
		&=	\sup_{\|\widetilde f\|_{M^1(\R^2)}=1} |\langle fg,{\widetilde f}\rangle| 
		= 	\sup_{\|\widetilde f\|_{M^1(\R^2)}=1}  |\langle f,{\widetilde f} \overline g \rangle |
		\leq 	\sup_{\|\widetilde f\|_{M^1(\R^2)}=1}  \|f\|_{M^\infty(\R)} \|{\widetilde f} \overline g \|_{M^1(\R^2)}\\
		&\leq 	\sup_{\|\widetilde f\|_{M^1(\R^2)}=1}  \|f\|_{M^\infty(\R)} \|{\widetilde f} \|_{M^1(\R^2)}  \|\overline g \|_{M^1(\R^2)}
	 	=	 \|f\|_{M^\infty(\R^2)} \| g \|_{M^1(\R^2)}\,.
\end{align*}
 
Note that with $\phi^\ast (t)=\overline {\phi (-t)}$, we have
\begin{align*}
 \int r(t) \phi(x-t) e^{-2\pi i t\xi}dt=V_{\phi^\ast}r (x,\xi),
\end{align*}
which is a bandlimited function since
\begin{align*}
 \iint V_{\phi^\ast}r (x,\xi)e^{2\pi i t\xi - x\nu} \, dx\, d\xi 
 				= \int r(t) \phi(x-t) e^{-2\pi i x\nu}dx
 				=r(t)\widehat{\varphi}(\nu)\,e^{-2\pi i t \nu}.  
\end{align*}
Using that  the $M^1$-norm is invariant under partial Fourier transforms and the equivalence between the $M^1$ and $L^1$ norms  which is  implied by the bandlimitation of  $P_{I_1+[-N(\epsilon),N(\epsilon)]}(x+q)\,  V_{\phi^\ast}r (x,\xi)$  to $(-1/T,1/T){\times}\{0\} +(-\delta,\Omega+\delta){\times}  (-\delta, T + \delta)$,
we obtain
\begin{align*}
  \Big\|  P_{I_1+[-N(\epsilon),N(\epsilon)]}(x+q)\,  r(t)\, 
				\phi(x-t) \Big\|_{M^1(\R^2)}
		& \asymp  \Big\| P_{I_1+[-N(\epsilon),N(\epsilon)]}(x+q)\,  V_{\phi^\ast}r (x,\xi) \Big\|_{M^1(\R^2)}\\
		& \asymp \Big\| P_{I_1+[-N(\epsilon),N(\epsilon)]}(x+q)\,  V_{\phi^\ast}r (x,\xi) \Big\|_{L^1(\R^2)} .
\end{align*}
Fix  $g \in \mathcal S (\R)$ and observe  that $\|V_g f\|_{L^p(\R^2)}$ defines a norm on $M^p(\R)$ equivalent to the $M^p(\R)$ norm given in \eqref{eqn:modulationspace} \cite{Groechenig01}. For  any $A\subset \R$  we obtain the uniform bound
\begin{align*}
 \| \sum_{nT\in A} c_n \delta_{nT} \|_{M^\infty(\R)} & \asymp \| V_g \sum_{nT\in A}c_n\delta_{nT}\|_{L^\infty(\R)} =\| \sum_{nT\in A}c_n g(nT-t)e^{2\pi i \nu nT}\|_{L^\infty(\R)}
 \\ & \leq  \| \sum_{nT\in A} |c_n|\, |g(nT-t)|\|_{L^\infty(\R)}  \leq  \| \sum_{n\in\Z} |c_n|\, |g(nT-t)|\|_{L^\infty(\R)}< \infty .
\end{align*}
The first norm inequality stems from the fact that for all $g\in M^1(\R)$, $\|V_g f\|_{L^p(\R^2)}$ defines a norm on $M^p(\R)$ equivalent to the $M^p(\R)$ norm given in \eqref{eqn:modulationspace}.

Combining this upper bound on $\|w- w_2 \|_{M^\infty(\R)}$ with the above estimate for $\|\sigma_{H_C}-\sigma_{\widetilde H}\|_{L^\infty(S)}$ and \eqref{eqn:D}, we conclude
 \begin{align*}
 \|\sigma_{H_C}&-\sigma_{\widetilde H}\|_{L^\infty(S)} \\
 	  &\lesssim 
		D \| {H_C} \|_{{\mathcal L}(M^\infty(\R))}   \frac{L^2T}{1-\epsilon}\,  \|   b_{jq}   \|_{\ell^\infty }\!
		 \sum_{qT\notin I_1+[-R(\epsilon),R(\epsilon)] }  \Big\| P_{I_1+[-N(\epsilon),N(\epsilon)]}(x+q)\,  V_{\phi^\ast}r (x,\xi) \Big\|_{L^1(\R^2)}
\\	&\leq  D \epsilon \|{H_C}\|_{{\mathcal L}(M^\infty(\R))} \asymp  D\epsilon \| \sigma_{H_C}\|_{L^\infty(\R^2)} \leq D\epsilon \| \sigma_{H}\|_{L^\infty(\R^2)}\leq D\epsilon\mu.
 \end{align*}
 Choosing $R(\epsilon)$ above large to yield $D$ small enough to compensate all the multiplicative constants, we obtain
\begin{equation*}
 \|\sigma_{H_C}-\sigma_{\widetilde H}\|_{L^\infty(S)} \leq \epsilon\mu.
\end{equation*}

As a meaningful statement is only obtained for $\epsilon<1$, this bound directly implies that 
\begin{equation*}
\|\sigma_{\widetilde H}\|_{L^\infty(\R^2)}\leq 2\mu.
\end{equation*}
Combining this with the bound 
\begin{equation*}
 \|\sigma_H-\sigma_{\widetilde H}\|_{L^\infty(\R^2)}\leq  \|\sigma_H-\sigma_{ H_C}\|_{L^\infty(\R^2)}+ \|\sigma_{H_C}-\sigma_{\widetilde H}\|_{L^\infty(\R^2)} \leq 2\epsilon\mu, 
\end{equation*}
Theorem~\ref{thm:localization1} directly yields the result with a constant of twice the size as in Theorem~\ref{thm:localization1}. \qed

\section{Reconstruction of bandlimited operators from discrete measurements}\label{sec:Op_id}
This section concerns the discrete representation given in Theorem ~\ref{prop:exp}. First, we prove this theorem, hence establishing that indeed this representation is globally exact.

\par{\em Proof of Theorem ~\ref{prop:exp}:}
%
%
%
%
The proof 
is similar to the proof of Theorem~\ref{thm:reconstruction}  given in \cite{PW12}.
The main idea is to use a Jordan domain argument to cover a fixed compact set $M$ of size less than one by shifts of a rectangle that still have combined area less than one and then to combine identifiability results for each of them to obtain identifiability for the whole set. Indeed,   there exist $L$ prime and $T,\Omega>0$ with $T\Omega=\frac{1}{L}$ such that
\begin{align*}
\supp(\eta) \subseteq \bigcup_{j=0}^{L-1} R + ( k_j T ,n_j\Omega)&\subseteq [-(L-1)T/2,(L+1)T/2]\times [-L\Omega /2,L\Omega/2]\notag \\ &= [-1/(2\Omega) +T/2,1/(2\Omega) +T/2] \times [-1/(2T),1/(2T)]
\end{align*}
 where $R=[0,T)\times [-\Omega/2,\Omega/2)$, and the sequence $(k_j, n_j) \in \Z^2$ consists of distinct pairs.
For $\delta > 0$ small enough (and possibly slightly smaller $T,\Omega$, and a larger prime $L$), one can even achieve
\begin{equation*}\label{eq:b1}
M_{\delta} \subseteq \bigcup_{j=0}^{L-1} R+( k_j T,n_j\Omega) \subseteq   [-(L-1)T/2,(L+1)T/2]\times [-L\Omega /2,L\Omega/2]
\end{equation*}
where $M_{\delta}$ 
is the $\delta$-neighborhood of $M$.

Fix such $\delta$ and let $r,\phi\in\mathcal{S}(\R)$ satisfy \eqref{eq:r_phi_1a} and \eqref{eq:r_phi_2a} for this $\delta$.
Clearly,
\begin{equation}\label{eq:r_phi_3}
(k,n)\neq (k_j, n_j) \text{ for all $j$  implies }  S_{\delta} \cap \Big( R+(kT,n\Omega) \Big) = \emptyset \text{ and } \eta(t,\gamma) r(t-kT) \widehat{\phi}(\gamma-n\Omega) = 0,
\end{equation}
a fact that we shall use below.

Define the identifier $
    \ident = \sum_{n\in\Z}c_n\, \delta_{nT}
$,
where $\{c_n\}$ is $L$-periodic and observe that
\begin{align*}
H\ident(x) &= \iint \eta(t,\gamma)\,e^{2\pi i\gamma x}\ident (x-t)\,dt\, d\gamma 
= \iint \boldsymbol{\eta}(t,\gamma)\,e^{2\pi i\gamma(x-t)}\sum_{k\in\Z}c_k\delta_{kT}(x-t)\,dt\, d\gamma \\
&= \sum_{k\in\Z}c_k \int \boldsymbol{\eta}(x-kT,\gamma)\,e^{2\pi i\gamma kT} d\gamma \\
&= \sum_{m\in\Z}\sum_{k=0}^{L-1} c_{k+p} \int \boldsymbol{\eta}(x-(mL+k+p)T,\gamma)\,e^{2\pi i\gamma (mL+k+p)T} d\gamma
\end{align*}
for any $p\in\Z$.
We shall use  the non-normalized Zak transform $Z_{LT}:L^2(\R)\longrightarrow L^2\big([0,LT)\times [-\Omega/2,\Omega/2)\big)$ defined by
\begin{equation*}
Z_{LT} f(t,\gamma) = \sum_{n\in\Z}f(t-nLT)\, e^{2\pi i  n L T \gamma}\,.
\end{equation*}
We compute using the Poisson summation formula and the fact that $\Omega=1/LT$
\begin{align*}
 (Z_{LT}& \circ H)\ident (t,\nu)
= \sum_{n\in\Z}H\ident (t-nLT)\,e^{2\pi i nLT\nu} \\
 &= \sum_{m,n\in\Z}e^{2\pi i T  n L  \nu} \sum_{k=0}^{L-1} c_{k+p} \int \boldsymbol{\eta}(t-(nL+mL+k+p)T,\gamma)\,e^{2\pi i\gamma (mL+k+p)T } d\gamma \\
 &= \sum_{k=0}^{L-1} c_{k+p} \sum_{m,n\in\Z}e^{2\pi i T nL \nu} \int \boldsymbol{\eta}(t-(mL+k+p)T,\gamma)\,e^{2\pi i\gamma T((m-n)L+k+p)} d\gamma \\
&= \sum_{k=0}^{L-1} c_{k+p} \sum_{m\in\Z} \int \boldsymbol{\eta}(t-(mL+k+p)T ,\gamma)\,e^{2\pi i\gamma (mL+k+p)T } \sum_{n\in\Z}e^{2\pi i  n L (\nu-\gamma)T} d\gamma \\
&= \sum_{k=0}^{L-1} c_{k+p} \sum_{m\in\Z} \int \boldsymbol{\eta}(t-(mL+k+p)T ,\gamma)\,e^{2\pi i\gamma (mL+k+p)T } \frac 1 {LT}\sum_{n\in\Z}\delta_{n/LT}(\nu-\gamma) d\gamma \\
& =\Omega \sum_{k=0}^{L-1} c_{k+p} \sum_{m,n\in\Z} \boldsymbol{\eta}(t-(mL+k+p)T ,\nu +  n\Omega)\,e^{2\pi i  (\nu+\Omega n)(mL+k+p)T}
\end{align*}
By (\ref{eq:r_phi_3}) we get for $p=0,\ldots, L-1$,
\begin{align*}
r(t)\widehat{\phi}(\nu)&(Z_{LT}\circ H)\ident(t+pT,\nu)
\\&= \Omega\sum_{k=0}^{L-1} c_{k+p} \sum_{m,n\in\Z} r(t)\widehat{\phi}(\nu)\boldsymbol{\eta}(t-(mL+ k)T,\nu+ n\Omega)
e^{2\pi i T (\nu+n\Omega)(mL+k+p)} \nonumber \\
&=\Omega \sum_{j=0}^{L-1} c_{p+k_j} r(t)\widehat{\phi}(\nu)\boldsymbol{\eta}(t+  k_j T,\nu+ n_j\Omega)\,e^{2\pi i  (\nu+ n_j\Omega)T(p+k_j)}. \nonumber \\
&=\Omega e^{2\pi i\nu  pT} \sum_{j=0}^{L-1} (T^{k_j} M^{n_j}c)_p \Big(e^{2\pi i\nu  k_j T}r(t)\widehat{\phi}(\nu)\, \boldsymbol{\eta}(t+ k_j T,\nu+ n_j\Omega) \Big) \, ,\nonumber
\end{align*}
 where here and in the following, $T:(c_0,c_1,\dots, c_{L-2},c_{L-1})\mapsto (c_{L-1},c_0,\dots, c_{L-3},c_{L-2})$ and $M:(c_0,c_1,\dots, c_{L-2},c_{L-1})\mapsto (e^{2\pi i 0 /L }c_0,\, e^{2\pi i 1 /L }c_1,\dots, \, e^{2\pi i (L-2)/L }c_{L-2},\, e^{2\pi i (L-1)/L }c_{L-1})$, that is, $(T^{k_j} M^{n_j}c)_p=e^{2\pi i \frac{n_j(p+k_j)}{L}}c_{p+k_j}$.
Equivalently, we obtain the matrix equation
\begin{align}\label{eq:a3}
& [e^{- 2\pi i\nu  pT} r(t)\widehat{\phi}(\nu)(Z_{LT}\circ H)\ident(t+pT,\nu)]_{p=0}^{L-1}  \\
& \hspace{1cm} =\Omega \mathbf{A} [e^{2\pi i\nu  k_j T}r(t)\widehat{\phi}(\nu)\boldsymbol{\eta}(t+ k_j T,\nu+ n_j\Omega)]_{j=0}^{L-1} \nonumber
\end{align}
 where $\mathbf{A}$ is a $L\times L$ matrix, whose $j$th column is $T^{k_j}M^{n_j}c \in\C^L$. 
$\mathbf{A}$ is a submatrix of the $L\times L^2$ marix $\mathbf{G}$, whose columns are $\{T^k M^lc\}_{k,l=0}^{L-1}$. It was shown in \cite{KPR08} that if $L$ is prime, then we can choose $c\in\C^L$ such that every $L\times L$ submatrix of $\mathbf{G}$ is invertible. In fact, the set of such $c\in\C^L$ is a dense open subset of $\C^L$ \cite{KPR08}. Hence we can apply the matrix $\mathbf{A^{-1}} =: [b_{jp}]_{j,p=1}^L$ on both sides of Equation (\ref{eq:a3}) to obtain
\begin{align}\label{eq:a4}
& e^{2\pi i\nu  k_j T}r(t)\widehat{\phi}(\nu)\boldsymbol{\eta}(t+ k_j T,\nu+ n_j\Omega) 
 = LT \sum_{p=0}^{L-1}b_{jp} e^{-2\pi i\nu p T}r(t)\widehat{\phi}(\nu)(Z_{LT}\circ H)\ident(t+pT,\nu) 
\end{align}
for every $j=0,1,\dots,L-1$.

In fact, until this point the proof agrees with the proof of \eqref{eqn:reconstructionformula} in Theorem~\ref{thm:reconstruction}. 
Indeed, if we extend $\{b_{jp}\}_p$ to a $L$-periodic sequence by setting $b_{j,p+mL} = b_{jp},$  replace the so far unused property \eqref{eq:r_phi_2a} by \eqref{eq:r_phi_2} then further computations \cite{Pfa10} give
\begin{align}
h(x,t) &= LT  \sum_{j=0}^{L-1} r(t- k_j T)\Big( \sum_{q\in\Z} b_{jq} H\ident(t-(k_j+q)T) \phi(x-t+(k_j+q)T)\Big) e^{2\pi i n_j \Omega (x-t)} \nonumber.
\end{align}

Observe that \eqref{eq:r_phi_2a} implies that $(r,T\Z\times \frac {\Omega L} {\beta_2}\Z)=\{\mathcal T_{kT}\mathcal M_{\ell L \Omega/{\beta_2}}r\}_{k,\ell\in\Z}$ is a tight Gabor frame whenever $\beta_2\geq 1 +2\delta /T$ as, in this case, $(r,\frac{\beta_2}{\Omega L}\Z\times\frac 1 T \Z)=(r,\beta_2 T\Z\times \Omega L\Z)$ is an orthogonal sequence and the Ron-Shen criterion applies \cite{Groechenig01,RS95a}. The same arguments imply that
$(\widehat\phi, \Omega\Z\times\frac {LT} {\beta_1}\Z)$ is a tight Gabor frame. Using a simple tensor argument, we obtain that
$\{\Psi_{m,n,l,k}\}_{m,n,l,k\in\Z}$ forms a tight Gabor frame where
\begin{align*}
\Psi_{m,n,l,k}(t,\nu) &= \mathcal T_{(kT,n\Omega)}\mathcal M_{L(\ell \Omega \beta_2 T ,T/\beta_1)}\ r{\otimes} \widehat{\phi}(t,\nu)\\ &=e^{2\pi i L(\frac{mT(\nu-n\Omega)}{\beta_1}+ \frac{\ell \Omega(t-kT)}{\beta_2})}\ r(t-kT)\ \widehat{\phi}(\nu-n\Omega)\,.\nonumber
\end{align*}
The frame bound is $T\Omega  L^2 T\Omega/(\beta_1\beta_2)=1/(\beta_1\beta_2)$.
We set $\Phi_{m,-n,l,-k}=\mathcal F_s \Psi_{m,n,l,k}$. Clearly, as $\mathcal F_s$ is unitary, we have that $\{\Phi_{m,n,l,k}\}_{m,n,l,k\in\Z}$ forms a tight frame with frame bound $1/(\beta_1\beta_2)$, in fact, a tight Gabor frame as
\begin{align}
\Phi_{m,n,l,k}(x,\xi) &= \mathcal F_s \Psi_{m,-n,l,-k} (x,\xi) \nonumber 
= (\mathcal F \mathcal T_{-kT} \mathcal M_{\ell L \Omega/\beta_2} r) (\xi)\ ( \mathcal F^{-1} \mathcal T_{-n\Omega} \mathcal M_{m T L/\beta_1} \widehat{\phi})(x) \nonumber \\
&= (  \mathcal M_{kT} \mathcal T_{\ell L \Omega/\beta_2} \widehat{r}) (\xi)\ (  \mathcal M_{n\Omega} \mathcal T_{m T L/\beta_1} \phi)(x)\nonumber \\
&= e^{2\pi i(nm+kl)/\lambda} (  \mathcal T_{\ell L \Omega/\beta_2}\mathcal M_{kT} \widehat{r}) (\xi)\ ( \mathcal T_{m T L/\beta_1} \mathcal M_{n\Omega} \phi)(x).\nonumber
\end{align}

Note that (\ref{eq:r_phi_3}) together with the fact that the symplectic Fourier transform is unitary implies that the coefficients in the Gabor frame expansion of $\boldsymbol{\sigma}$ satisfy 
$$ \langle \boldsymbol{\sigma},\Phi_{m,-n_j,l,-k_j}\rangle=\langle \boldsymbol{\eta},\Psi_{m,n,l,k} \rangle = 0 \textnormal{  unless  } (n,k)= (n_j,k_j) \textnormal{ for some } j.$$ 
Hence we need to estimate $\sigma^{(j)}_{m,\ell}= \langle \boldsymbol{\sigma},\Phi_{m,-n_j,l,-k_j}\rangle$ for $j=0,1,\dots,L-1$. We obtain by (\ref{eq:a4})

\begin{align}
\sigma^{(j)}_{m,\ell}&= \langle \boldsymbol{\sigma},\Phi_{m,-n_j,l,-k_j}\rangle
= \langle \boldsymbol{\eta},\Psi_{m,n_j,l,k_j}\rangle \nonumber \\
&=\iint \boldsymbol{\eta}(t,\nu) e^{-2\pi i L(\frac{Tm(\nu-n_j\Omega)}{\beta_1}+ \frac{\ell \Omega(t- k_j T)}{\beta_2})}\ r(t- k_j T)\ \widehat{\phi}(\nu-n_j\Omega) dtd\nu \nonumber \\
&=\iint \ r(t)\ \widehat{\phi}(\nu) \boldsymbol{\eta}(t+ k_j T,\nu+n_j\Omega) e^{2\pi i \nu  k_j T}\ \ e^{-2\pi i (L(\frac{m \nu T}{\beta_1}+ \frac{\ell t \Omega }{\beta_2}) +\nu  k_j T}dtd\nu \nonumber
\\
&=\iint \ LT \sum_{p=0}^{L-1}b_{jp} e^{-2\pi i\nu p T}r(t)\widehat{\phi}(\nu)(Z_{LT}\circ H)\ident(t+pT,\nu) \ e^{-2\pi i (L(\frac{m \nu T}{\beta_1}+ \frac{\ell t \Omega }{\beta_2}) +\nu  k_j T)}dtd\nu \nonumber \\
&=LT \sum_{p=0}^{L-1} b_{jp} \iint \ r(t) \widehat{\phi}(\nu)\ e^{-2\pi i\nu p T} (Z_{LT}\circ H)\ident(t+pT,\nu) \ e^{-2\pi i (L(\frac{m \nu T}{\beta_1}+ \frac{\ell t \Omega }{\beta_2}) +\nu  k_j T)}dtd\nu \nonumber\\
&=LT \sum_{p=0}^{L-1} b_{jp} \iint \ r(t) \widehat{\phi}(\nu)\ e^{-2\pi i\nu p T} \sum_{q\in\Z} H\ident(t+pT-qLT) \ e^{2\pi i \nu qLT} e^{-2\pi i (L(\frac{m \nu T}{\beta_1}+ \frac{\ell t \Omega }{\beta_2}) +\nu  k_j T}dtd\nu \nonumber \\
&=LT \sum_{p=0}^{L-1} b_{jp} \sum_{q\in\Z} \Big(\int r(t) H\ident (t+pT-qLT) e^{-2\pi i L\frac{\ell t \Omega }{\beta_2}} dt\Big)\Big(\int \widehat{\phi}(\nu)\ e^{2\pi i\nu T(qL- p -k_j-mL/\beta_1)}\,d\nu\Big) \nonumber
\\
&=LT  \sum_{q\in\Z} b_{jq} \Big(\int r(t) H\ident(t+qT) e^{-2\pi i L\frac{\ell t \Omega }{\beta_2}} dt\Big)\Big(\int \widehat{\phi}(\nu)\ e^{2\pi i\nu T(-q -k_j-mL/\beta_1)}\,d\nu\Big) \nonumber
\\
&=LT  \sum_{q\in\Z} b_{jq}\phi (T(-q -k_j-mL/\beta_1)) \Big(\int H\ident (t) e^{-2\pi i L\frac{\ell \Omega(t-qT)}{\beta_2}} r(t-qT)dt\Big)\nonumber
\\
&=LT  \sum_{q\in\Z} b_{jq}\phi (T(-q -k_j-mL/\beta_1))\  \langle H\ident, \mathcal T_{qT}\mathcal M_{\ell L \Omega/\beta_2}\, r\rangle,  \nonumber
\end{align}
where $ b_{jq}=b_{jq'}$ for $q=mL+q'$ with $q'=0,1, \ldots, L-1$. We can hence set
\begin{equation*}
C_{q,l}(H\ident) = \langle H\ident, \mathcal T_{qT}\mathcal M_{\ell L \Omega/\beta_2}\, r\rangle.
\end{equation*}
To sum up, 
\begin{align}
\boldsymbol{\sigma}(x,\xi)
&= \frac{1}{\beta_1\beta_2}\sum_{j=0}^{L-1}\sum_{m,\ell\in\Z} \langle \boldsymbol{\sigma},\Phi_{m,-n_j,l,-k_j} \rangle \Phi_{m,-n_j,l,-k_j}(x,\xi) \nonumber \\
&= \frac{LT}{\beta_1\beta_2}\sum_{j=0}^{L-1} e^{-2\pi i(xn_j\Omega + \xi  k_j T )} \sum_{m,\ell\in\Z} \sigma^{(j)}_{m,\ell} \ \widehat{r}\big(\xi-\frac{\ell L \Omega}{\beta_2}\big)\ \phi\big(x - \frac{m T L}{\beta_1}\big)\,, \label{eq:03}
\end{align}
where
\begin{align*}
\sigma^{(j)}_{m,\ell} = \sum_{q\in\Z}\  b_{jq}\ \phi (a(-q -k_j-mL/\beta_1))\ C_{q,l}(Hw).
\end{align*}

Applying the symplectic Fourier transform to  \eqref{eq:03} yields
\begin{align}
\eta(t,\nu)&=e^{-2\pi i \nu t}\boldsymbol{\eta}(t,\nu)\nonumber
= e^{-2\pi i \nu t}\, \frac{LT}{\beta_1\beta_2}\sum_{j=0}^{L-1} \sum_{m,\ell\in\Z} \sigma^{(j)}_{m,\ell}\  \mathcal F_s \ \Big( \mathcal M_{(-n_j\Omega, - k_j T )} \ \mathcal T_{(\frac{m T L}{\beta_1},\frac{\ell L \Omega}{\beta_2})}\phi{\otimes}\widehat{r} \Big) (t,\nu)\nonumber
\\
&= e^{-2\pi i \nu t}\ \frac{LT}{\beta_1\beta_2}\sum_{j=0}^{L-1} \sum_{m,\ell\in\Z} \sigma^{(j)}_{m,\ell}\  \mathcal T_{(  k_j T ,-n_j\Omega)} \ \mathcal M_{(\frac{\ell L \Omega}{\beta_2},-\frac{m T L}{\beta_1})}\   r{\otimes}\widehat{\phi}\,(t,\nu) \nonumber
\\
&=  \frac{LT}{\beta_1\beta_2}\sum_{j=0}^{L-1} \sum_{m,\ell\in\Z} \sigma^{(j)}_{m,\ell}\  \mathcal T_{(  k_j T ,-n_j\Omega)} \ \mathcal M_{(\frac{\ell L \Omega}{\beta_2},-\frac{m T L}{\beta_1})}\   \Big(r{\otimes}\widehat{\phi}\,(t,\nu)\, e^{-2\pi i (\nu+n_j\Omega)( t-k_jT)}\Big), \nonumber
\\
&=  \frac{LT}{\beta_1\beta_2}\sum_{j=0}^{L-1} \sum_{m,\ell\in\Z} \sigma^{(j)}_{m,\ell}\,e^{2\pi i n_j\Omega k_j T}\  \mathcal T_{(  k_j T ,-n_j\Omega)} \ \mathcal M_{(\frac{\ell L \Omega}{\beta_2}-n_j\Omega,\, k_jT-\frac{m T L}{\beta_1})}\  \Big( r{\otimes}\widehat{\phi}\,(t,\nu)\, e^{-2\pi i \nu t}\Big). \nonumber
\end{align}
For $U(t,\nu)= r{\otimes}\widehat{\phi}\,(t,\nu)\, e^{-2\pi i \nu t}$, we have
\begin{align*}
\mathcal F_s U (x,\xi)&=  \iint r(t) \widehat{\phi}(\nu) e^{-2\pi i \nu t } e^{ -2\pi i (\xi t-\nu x) } \, d\nu \, dt \\
    &=\int r(t) \phi(x-t)e^{ -2\pi i \xi t } \, dt=\int r(t) \overline{\phi}(t-x)e^{ -2\pi i \xi t } \, dt=V_\phi r(x,\xi),
\end{align*}
where we used that $\widehat{\phi}$ real valued implies $\phi(y)=\overline \phi(-y)$. Now, we compute
\begin{align}
\sigma(x,\xi)&= \mathcal F_s \eta\,(x,\xi) \nonumber 
\\
&=  \frac{LT}{\beta_1\beta_2}\sum_{j=0}^{L-1} \sum_{m,\ell\in\Z} \sigma^{(j)}_{m,\ell}\,e^{2\pi i n_j\Omega k_j T}\  \mathcal F_s \ \Big( \mathcal T_{(  k_j T ,-n_j\Omega)} \ \mathcal M_{(\frac{\ell L \Omega}{\beta_2}-n_j\Omega,\, k_jT-\frac{m T L}{\beta_1})}\  U\Big)\, (x,\xi), \nonumber
\\
&=  \frac{LT}{\beta_1\beta_2}\sum_{j=0}^{L-1} \sum_{m,\ell\in\Z} \sigma^{(j)}_{m,\ell}\,e^{2\pi i n_j\Omega k_j T}\  \ \mathcal M_{(-n_j\Omega,- k_j T)} \ \mathcal T_{(\frac{m T L}{\beta_1}-k_jT, \, \frac{\ell L \Omega}{\beta_2}-n_j\Omega)}\,  V_\phi r\, (x,\xi), \nonumber
\\
&= \frac{LT}{\beta_1\beta_2}\sum_{j=0}^{L-1} e^{-2\pi i(xn_j\Omega + \xi  k_j T )} e^{2\pi i n_j\Omega k_j T}\sum_{m,\ell\in\Z} \sigma^{(j)}_{m,\ell} \ V_\phi r(x - \frac{m T L}{\beta_1}+k_jT, \ \xi-\frac{\ell L \Omega+n_j\Omega}{\beta_2}\big).\label{eq:03B}
\end{align}
The convergence in  \eqref{eq:03} and  \eqref{eq:03B} is defined in the weak sense, but can be shown to converge absolutely and  uniformly on compact subsets. \qed

Next we prove Theorem~\ref{thm:localsamp}, that is, the direct local correspondence between the discretization values and the operator action.

{\it Proof of Theorem~\ref{thm:localsamp}.} \ 
We intend to apply Theorems~\ref{thm:localization1} and \ref{prop:exp}. We assume that the set $M$ as well as its enclosing rectangular grid are fixed, hence also the parameters $T$, $\Omega$, and $L$. The dependence of the constants, auxiliary functions, etc., in the following derivations on these parameters will be suppressed for notational convenience; this should be seen as analogue to the one-dimensional scenario where the arising constants also depend on the shape and not just the size of the frequency support. Furthermore, set $Q=\max(LT, L\Omega)$.

We can bound using \eqref{eqn:iso}
\begin{equation}\label{eq:coeffbd}
\begin{aligned}
 |\sigma^{(j)}_{m,\ell}| &= |\langle {\boldsymbol \sigma}, \Phi_{m, -n_j, \ell, -k_j}\rangle|
\leq \|\boldsymbol \sigma\|_\infty \|\Phi_{m, -n_j, \ell, -k_j}\|_1
\leq B\|\sigma\|_\infty   \|\hat r \otimes \phi \|_1
\leq \tilde B \mu 
\end{aligned}
\end{equation}
For the second inequality, we used that the $L_1$-norm is invariant under translations and modulations.

Furthermore, note that $V_\phi r \in {\mathcal{S}}(\R^2)$, so there is a decreasing positive function $\rho\in\S([0,\infty))$ such that for $\widetilde\rho(x,\xi)=\rho\Big(|x|\Big)\rho\Big(|\xi|\Big)$ one has  $|V_\phi r |\leq \frac{1}{8\tilde C T}\widetilde \rho$ pointwise. 

Now observe that, as $\rho$ is decreasing,
\begin{equation*}
 \sum\limits_{j=0}^\infty \alpha \rho(\alpha j)\leq \rho(0)+\sum\limits_{j=1}^\infty \int\limits_{\alpha(j-1)}^{\alpha j} \rho(t)dt =\|\rho\|_1+\|\rho\|_\infty.
\end{equation*}

We use this estimate to bound for arbitrary $(x,\xi)$
\begin{align*}
 |\tilde \sigma (x&,\xi)|\\=&\Big|\frac{LT}{\beta_1 \beta _2} \sum_{j=0}^{L-1} e^{-2\pi i(xn_j\Omega + \xi  k_j T )} e^{2\pi i n_j\Omega k_j T} \hspace{-.6cm} \sum_{(mLT/\beta_1,\ell L\Omega/\beta_2)\in S}  \hspace{-.6cm} \sigma^{(j)}_{m,\ell} V_\phi r(x - \big(\frac{m  L}{\beta_1}+k_j\big) T, \ \xi-\big(\frac{\ell L }{\beta_2}+n_j\big)\Omega\big) \Big| \\
\leq& \frac{LT}{\beta_1 \beta _2}  \sum_{(mLT/\beta_1,\ell L\Omega/\beta_2)\in S} |\sigma^{(j)}_{m,\ell}| \Big|V_\phi r(x - \big(\frac{m  L}{\beta_1}+k_j\big) T, \ \xi-\big(\frac{\ell L }{\beta_2}+n_j\big)\Omega\big)\Big| \\
\leq &  \frac{L T}{\beta_1 \beta _2}\sum_{j=0}^{L-1} \sum_{(mLT/\beta_1,\ell L\Omega/\beta_2)\in S} \frac{\tilde C \mu}{8\tilde C T}\  \rho\Big(\Big| x - \big(\frac{m  L}{\beta_1}+k_j\big) T\Big|\Big) \ \rho\Big(\Big|\xi-\big(\frac{\ell L }{\beta_2}+n_j\big)\Omega\big)\Big|\Big)\\
\leq &  \frac{\mu}{8 L\Omega T}\sum_{j=0}^{L-1} 4 \sum_{m,\ell=0}^\infty \  \frac{L T}{\beta_1}\rho\Big(\frac{m  L}{\beta_1} T\Big) \ \frac{L \Omega}{\beta _2}\rho\Big(\frac{\ell L }{\beta_2}\Omega\Big)
\leq  \frac{ \mu}{2} \Big(\|\rho\|_1+\|\rho \|_\infty\Big)^2
\end{align*}
and hence
\begin{equation*}
  \|\sigma-\tilde\sigma\|_\infty  \leq \|\sigma\|_\infty 
  	+\|\tilde\sigma\|_\infty \leq \mu+\frac{ \mu}{2} \Big(\|\rho\|_1+\|\rho \|_\infty\Big)^2=:C_1\mu.
\end{equation*}

By the definition of $\S$, for every $\delta>0$, there is a constant $C(\delta)$ such that for any fixed $0\leq j< L$,
\begin{align}
{\delta}\geq& 8\|\rho\|_{L^1(\R^+)}\|\rho\|_{L^1[C(\delta)-2Q,\infty)}\geq 8\|\widetilde\rho\|_{L_1(([-C(\delta)+Q, C(\delta)-Q]^2)^c)}\label{eq:schwartz},
\end{align}
and hence, for $(x,\xi)\in S-B(C(\delta))$,
\begin{align}
\delta \geq& 
 \frac{L^2}{\beta_1\beta_2}\sum\limits_{\substack{\ell,m \in\Z\\ \Big(x - \frac{m  L}{\beta_1} T, \xi-\frac{\ell L }{\beta_2}\Omega\Big) \notin [-C(\delta), C(\delta)]^2
 }}\rho\Big(\Big|x - \Big(\frac{m  L}{\beta_1}+k_j\Big) T\Big|\Big)\rho\Big(\Big|\xi-\Big(\frac{\ell L }{\beta_2}+n_j\Big)\Omega\Big|\Big).\label{eq:samplingtails}
\end{align}
To obtain \eqref{eq:samplingtails} from \eqref{eq:schwartz}, the boundary term in the discretization of the integral and the shifts by $k_j$ and $n_j$, respectively, are each compensated by increasing the dimensions of the integration/summation domain by $LT$ and $L\Omega$ in time and frequency, respectively, both of which are bounded by $Q$.

Note furthermore that, as $(x, \xi)\in S-B(C(\delta))$, a necessary condition for 
\begin{equation*}
\Big(x - \frac{m  L}{\beta_1} T, \xi-\frac{\ell L }{\beta_2}\Omega\Big) \notin [-C(\delta), C(\delta)]^2 
\end{equation*}
is that
\begin{equation*}
 {(mLT/\beta_1,\ell L\Omega/\beta_2)\notin S}.
\end{equation*}

Thus, using \eqref{eq:coeffbd} and the triangle inequality, we can bound \eqref{eq:samplingtails} from below obtaining
\begin{equation}\delta\mu \geq \Big|\frac{L^2 T}{\beta_1\beta_2}\sum\limits_{  {(mLT/\beta_1,\ell L\Omega/\beta_2)\notin S}} \sigma^{(j)}_{m,\ell} V_\phi r\Big(x - \Big(\frac{m  L}{\beta_1}+k_j\Big) T,\ \xi-\Big(\frac{\ell L }{\beta_2}+n_j\Big)\Omega\Big)\Big|. \label{eq:smallstf}
\end{equation}

Hence forming a weighted average (with complex weighting factors of modulus one) of Equation~\eqref{eq:smallstf} over the $L$ choices of $j$, we obtain
\begin{align*}
   \delta\mu \geq& \Big|\frac{LT}{\beta_1\beta_2}\sum_{j=0}^{L-1}  \!e^{-2\pi i(xn_j\Omega + \xi  k_j T )} e^{2\pi i n_j\Omega k_j T} \!\!\!\!\sum\limits_{   {(\frac{mLT}{\beta_1},\frac{\ell L\Omega}{\beta_2})\notin S}} \!\!\!\! \sigma^{(j)}_{m,\ell} V_\phi r\Big(x - \Big(\frac{m  L}{\beta_1}\!+\!k_j\Big) T,\ \xi-\Big(\frac{\ell L }{\beta_2}\!+\!n_j\Big)\Omega\Big)\Big|\\
=&|\sigma(x,\xi)-\tilde\sigma(x,\xi)|.
\end{align*}
This yields $\|\sigma-\tilde\sigma\|_{L_\infty(S-B(C(\delta)))}\leq \delta\mu$. Hence by Theorem~\ref{thm:localization1}, we conclude that
\begin{equation*}
 \|H f-\tilde H f\|_2\leq C \frac{\delta}{C_1} \mu
\end{equation*}
for all functions $f$ which are $\frac{\delta}{C_1}$-time-frequency-localized to $S-B(C(\delta))-B(d(\epsilon))$.
The result follows by choosing $\delta= \min\Big(\frac{C_1}{C}\epsilon, C_1\epsilon\Big)$ and $D(\epsilon)=C(\delta)+d(\epsilon)$. \qed

\subsection*{Acknowledgments}
The authors thank the anonymous referee for the constructive comments, which greatly improved the paper, and Onur Oktay, who participated in initial discussions on the project.
Part of this research was carried out during a sabbatical of G.E.P.  and a stay of F.K. at the Department of Mathematics and the Research Laboratory for Electronics at the Massachusetts Institute of Technology.
Both are grateful for the support and the stimulating research environment.
F.K. acknowledges support by the Hausdorff Center for Mathematics, Bonn.
G.E.P. acknowledges funding by the German Science Foundation (DFG) under Grant 50292 DFG PF-4, Sampling Operators.



\bibliographystyle{amsplain}
\bibliography{OpSigDel}
\end{document}